\documentclass[12pt, twoside]{article}
\usepackage[T1, OT1]{fontenc}
\usepackage{hyperref}
\usepackage{authblk, lipsum}
\usepackage{amssymb}
\usepackage{amsmath, enumerate}
\usepackage{amsthm}
\usepackage{color}
\usepackage{mathrsfs}
\usepackage{txfonts}

\allowdisplaybreaks

\def\cec{{\ce_{c,\tau}}}

\def\rr{{\mathbb R}}
\def\rn{{{\rr}^n}}
\def\zz{{\mathbb Z}}

\def\nn{{\mathbb N}}
\def\hh{{\mathbb H}}
\def\ff{{\mathbb F}}

\def\pp{{\mathbb P}}
\def\qq{{\mathbb Q}}

\def\ca{{\mathcal A}}

\def\ce{{\mathcal E}}
\def\cf{{\mathcal F}}
\def\cg{{\mathcal G}}

\def\cj{{\mathcal J}}

\def\cq{{\mathcal Q}}

\def\ct{{\mathcal T}}

\def\fz{\infty}
\def\az{\alpha}
\def\bz{\beta}
\def\dz{\delta}

\def\eps{\varepsilon}

\def\kz{\kappa}

\def\sz{\sigma}

\def\lf{\left}
\def\r{\right}
\def\lfz{{\lfloor}}
\def\rfz{{\rfloor}}
\def\la{\langle}
\def\ra{\rangle}

\def\wz{\widetilde}
\def\wh{\widehat}
\def\st{\subset}

\def\bh{\backslash}

\def\supp{\mathop\mathrm{\,supp\,}}

\newtheorem{theorem}{Theorem}[section]
\newtheorem{lemma}[theorem]{Lemma}
\newtheorem{proposition}[theorem]{Proposition}
\newtheorem{corollary}[theorem]{Corollary}
\theoremstyle{definition}
\newtheorem{definition}[theorem]{Definition}

\numberwithin{equation}{section}

\begin{document}
\arraycolsep=1pt

\title{\bf Fourier dimension and avoidance of linear patterns
\footnotetext{\hspace{-0.35cm} 2010 {\it
Mathematics Subject Classification}.
Primary 42A32; Secondary 42A38.
\endgraf {\it Key words and phrases}.
Salem set, Fourier dimension
\endgraf
}}
\author{Yiyu Liang and Malabika Pramanik}  
\date{ }
\maketitle
\newcommand{\Addresses}{{
  \bigskip
  \footnotesize

  Y.~Liang, \textsc{Department of Mathematics, Beijing Jiaotong University, No.3 Shangyuancun, Haidian District,  Beijing 100044, P.~R.~China}\par\nopagebreak
  \textit{E-mail address}: \texttt{yyliang@bjtu.edu.cn}

  \medskip

  M.~Pramanik, \textsc{Department of Mathematics, 1984 Mathematics Road, University of British Columbia, Vancouver, Canada V6T 1Z2} \par\nopagebreak
  \textit{E-mail address}: \texttt{malabika@math.ubc.ca}

%

}}

\vspace{-0.8cm}

\begin{center}
\begin{minipage}{13cm}
{\small {\bf Abstract}\quad
The results in this paper are of two types. On one hand, we construct sets of large Fourier dimension that avoid nontrivial solutions of certain classes of linear equations. In particular, given any finite collection of translation-invariant linear equations of the form 
\begin{equation}
\sum_{i=1}^v m_ix_i=m_0x_0, \; \text{ with } (m_0, m_1, \cdots, m_v) \in \mathbb N^{v+1},  m_0 = \sum_{i=1}^{v} m_i \text{ and } v \geq 2,  \label{rational-eqn} 
\end{equation}
we find a Salem set $E \subseteq [0,1]$ of dimension 1 that contains no nontrivial solution of any of these equations; in other words, there does not exist a vector $(x_0, x_1, \cdots, x_v) \in E^{v+1}$ with distinct entries that satisfies any of the given equations. Variants of this construction can also be used to obtain Salem sets that avoid solutions of translation-invariant linear equations of other kinds, for instance, when the collection of linear equations to be avoided is uncountable or has irrational coefficients. While such constructions seem to suggest that Salem sets can avoid many configurations, our second type of results offers a counterpoint. We show that a set in $\mathbb R$ whose Fourier dimension exceeds $2/(v+1)$ cannot avoid nontrivial solutions of all equations of the form \eqref{rational-eqn}. In particular, a set of positive Fourier dimension must contain a nontrivial linear pattern of the form \eqref{rational-eqn} for some $v$, and hence cannot be rationally independent. This is in stark contrast with known results \cite{M17} that ensure the existence of rationally independent sets of full Hausdorff dimension. The latter class of results may be viewed as quantitative evidence of the structural richness of Salem sets of positive dimension, even if the dimension is arbitrarily small.         
%
}
\end{minipage}
\end{center}

\tableofcontents
\section{Introduction}
Many problems in geometric measure theory are concerned with the following question: ``Can large sets avoid many patterns?'' Stated in this level of generality, the question lacks precision, both in the quantification of size and in the specification of patterns. ``Large" could be interpreted either as large cardinality, nontrivial Lebesgue measure, positive asymptotic or  Banach density, large Hausdorff, Minkowski or Fourier dimension. ``Patterns" could be geometric in nature, for example, arithmetic or geometric progressions, equilateral triangles, parallelograms; alternatively, they could be algebraic, such as solutions of certain equations. Regardless of the many possible variants of such a question, it would seem that a natural answer would be ``no'', with any reasonable definition. Indeed, there is a large body of work that supports this intuition; see \cite{{BIT-2016}, {CLP}, {GI-2012}, {GILP}, {GIP}, {HLP}, {LP09}}. 
\vskip0.1in
\noindent However, there are also many results in the literature that challenge this intuition, especially when slight variations in the notions of size lead to very different conclusions regarding the existence of patterns. For example, in the discrete setting, a classical result of  
Behrend \cite{B46} says that for any $\eps>0$ and all sufficiently large positive integers $M$, there exists a set $X_M\subseteq [M] := \{0,1, 2, \ldots,M-1\}$
such that $\#(X_M)>M^{1-\eps}$ 
and $X_M$ contains no nontrivial three-term arithmetic progression. This is in sharp contrast with the celebrated results of Roth \cite{{Roth-1}, {Roth-2}} and Szemer\'edi \cite{{Szemeredi-1}, {Szemeredi-2}}, which state that for any $k \geq 3$ and any $c > 0$, there exists $M_0 \geq 1$ such that for $M \geq M_0$, any set $X_M \subseteq [M]$ obeying $\#(X_{M}) \geq cM$ contains a nontrivial $k$-term arithmetic progression. 
\vskip0.1in
\noindent Similar results exist in the continuum as well.  For instance, one can deduce from the Lebesgue density theorem that any set in $\mathbb R$ with a Lebesgue density point contains a nontrivial affine copy of any finite configuration. This conclusion applies therefore to any set of positive Lebesgue measure. On the other hand, Keleti \cite{K98} constructs a compact subset $E \subseteq [0,1]$
with Hausdorff dimension 1 but Lebesgue measure zero 
such that there does not exist any nontrivial solution of $x-y=z-w$, with $x< y \leq z < w$ and $(x,y,z,w) \in E^4$. In particular, $E$ avoids all three-term arithmetic progressions. Many subsequent results \cite{{DPZ}, {FP18}, {HKKMMMS}, {K98}, {K08}, {Maga}, {M17}} have explored the issue of avoidance further, providing examples of sets of large Hausdorff dimension that omit increasingly general families of algebraic and geometric patterns. Let us recall from \cite[Theorem 8.8]{Mattila-Book1} or \cite[Section 4.1]{Falconer} that the {\em{Hausdorff dimension}} $\text{dim}_{\mathbb H}(A)$ of a Borel set $A \subseteq \mathbb R^n$ is the supremum of exponents $\alpha > 0$ with the following property:  
there exists a probability measure $\mu$ supported on $A$ such that for some positive, finite constant $C_1$, 
\begin{equation}  \mu(B(x,r))\le C_1 r^{\alpha} \quad \mbox{ for all } x\in\rn, \; r>0. \label{ball-condition}  \end{equation} 

%
%
\vskip0.1in
\noindent On the other hand, the situation is expected to be different for sets $A$ of large {\em{Fourier dimension}}. The Fourier dimension $\text{dim}_{\mathbb F}(A)$ of a Borel set $A \subseteq \mathbb R^n$ is defined as the supremum of exponents $\beta \leq n$ obeying the following condition: there exists  a probability measure $\mu$ supported on $A$ and a positive finite constant $C_2$ such that 
\begin{equation} 
|\wh \mu(\xi)|\le C_2 (1 + |\xi|)^{- \beta/2}  \text{ for all } \xi \in\rn, \quad \text{ where } \; \wh{\mu}(\xi) := \int e^{-ix\xi} d\mu(x).  
 \label{Fourier-decay-condition} 
\end{equation} 
This expectation is heuristically based on Frostman's lemma \cite[p.~168]{Mattila-Book1}, which states that $\dim_\ff(A)\le\dim_\hh(A)$ for any Borel set $A$. This inequality implies that sets of large Fourier dimension form a smaller sub-class of sets of large Hausdorff dimension. It gives rise to the intuition that such sets are more likely to enjoy additional properties; in particular they could possibly contain a richer class of patterns. 
A Borel set whose Fourier dimension equals its Hausdorff dimension is called a {\it Salem set}.
\vskip0.1in
\noindent The intuition that large Salem sets are richer in structure than their non-Salem counterparts of the same dimension is perhaps also rooted in the known examples of such sets. Salem sets are ubiquitous among random sets. Many random constructions yields sets that are on one hand, often (almost surely) Salem, and on the other embody verifiable algebraic and geometric structures. 
The first such random construction is due to Salem himself \cite{Salem}; many subsequent random constructions have appeared in \cite{{Kahane-1}, {Kahane-2}, {Bluhm-1}, {Bluhm-2}, {LP09}, {Ekstrom}, {Shmerkin-Suomala}, {Chen}, {Chen-Seeger}}.
Deterministic examples of Salem sets are comparatively fewer \cite{{Jarnik-1}, {Jarnik-2}, {Kaufman}, {Hambrook-1}, {Hambrook-2}}, but they arise naturally in number theory \cite{{Besicovitch}, {Bovey-Dodson}, {Eggleston}} and are rich in arithmetic patterns as well. The work of K\"orner \cite{{Korner1},{Korner2}}, which explicitly addresses the relation between the rate of decay of the Fourier transform of a measure and possible algebraic relations within its support, is perhaps closest to the main focus of our article. 
\vskip0.1in
\noindent In this paper, we will provide a quantitative formulation of the heuristic principle that Salem sets possess richer structure, in the specific context of translation-invariant linear patterns. More precisely, we will be concerned with algebraic patterns that occur as a nontrivial zero of some function in the class  
\begin{align} &\hskip1.5in \cf  = \cf(\mathbb N) := \bigcup_{v=2}^{\infty} \cf_v(\mathbb N), \text{ where } \label{def-cF} \\ \cf_v(\mathbb N) &:=\lf\{f(x_0, \ldots, x_v) :=m_0x_0-\sum_{i=1}^v m_ix_i \; \vline 
\; \begin{aligned} &m_0, \cdots, m_v \in \mathbb N, \; m_0=\sum_{i=1}^v m_i, \\ &\text{gcd}(m_0, m_1, \cdots, m_v) = 1\end{aligned} \r\}. \label{def-Fn} \end{align} 
Here $v \in\mathbb N\bh\{1\}$ and $\mathbb N:=\{1,2,\ldots\}$. 
\begin{definition} The following definitions will be used throughout the article. \label{definition-pattern} 
\begin{itemize} 
\item Given $f \in \cf_{v}(\nn)$, a vector $x = (x_0, x_1, \ldots, x_v) \in \mathbb R^{v+1}$ is said to be a {\em{zero of $f$}} if it obeys the equation $f(x_0, \ldots, x_v) = 0$. Such a vector $x$ will also be referred to as a {\em{solution}} \label{def-zero} of the equation $f(x_0, \cdots, x_v)  = 0$. 
\item A zero $x = (x_0, \ldots, x_v) \in \mathbb R^{v+1}$ of a function $f \in \cf_v(\nn)$ is said to be {\em{nontrivial}} if the entries of $x$ are all distinct. All other zero vectors of $f$ are called {\em{trivial}}. These terms apply to solutions of equations of the form $f = 0$ as well. \label{def-nontrivial-zero} 
\item Given a set $E \subseteq \mathbb R$, we say that $E$ contains {\em{a nontrivial zero of $f \in \cf_\nu(\nn)$}} provided there exists $x = (x_0, x_1, \ldots, x_v) \in E^{v+1}$ with all distinct entries such that $f(x) = 0$. If no such $x \in E^{v + 1}$ exists, we say that $E$ {\em{avoids all nontrivial zeros of $f$}}. 
\item A set $E \subseteq \mathbb R$ is said to contain {\em{a nontrivial translation-invariant rational linear pattern}} provided it contains a nontrivial zero of some $f \in \cf$.  
\end{itemize}
\end{definition}
\noindent A three-term arithmetic progression $(x_0, x_1, x_2)$ with nonzero common difference is a simple example of a nontrivial translation-invariant rational linear pattern, since it is a nontrivial zero of the function $f(x_0, x_1, x_2) = 2x_0 - (x_1 + x_2)$. If a vector $x = (x_0, \cdots, x_v) \in \mathbb R^{v+1}$ is a trivial zero of some $f \in \mathcal F_{v} (\nn)$ with $v \geq 3$ but has at least two distinct entries, then the vector  $y = (y_0, \cdots, y_{\nu'})$ consisting of the distinct entries of $x$ provides a nontrivial zero of some $g \in \mathcal F_{\nu'}(\nn)$, $\nu' < \nu$.     
 \vskip0.1in
\noindent In the following subsection, we provide answers to variants of the following question: {\em{given $\cf^{\ast} \subseteq \cf(\mathbb N)$, how large a set $E \subset \mathbb R$, in the sense of Fourier dimension, can one construct that avoids all the nontrivial zeros of all $f \in \cf^{\ast}$? Alternatively, are sets of large enough Fourier dimension guaranteed to contain a nontrivial zero of some $f \in \cf^{\ast}$?}} The requirement $m_0 = \sum_{i=1}^{v} m_i$ in $\mathcal F_v(\mathbb N)$ is designed to avoid trivial answers; without this assumption, one can always find an avoiding interval (of positive Lebesgue measure) centred around 1. 
\subsection{Statement of results} 
We begin by providing the background that led to this work. In \cite[Theorem 1.2]{LP09},  \L aba and the second author show that if a compact set $A \subseteq [0,1]$ supports a probability measure $\mu$ obeying a ball condition of the type \eqref{ball-condition} and a Fourier decay condition of the type \eqref{Fourier-decay-condition}, then $A$ contains a nontrivial three term  arithmetic progression, provided (a) $\beta > 2/3$, (b) the constants $C_1$ and $C_2$ are appropriately controlled, and (c) the exponent $\alpha$ is sufficiently close to 1. 
The article \cite[Section 7]{LP09} also contains a large class of examples of Salem sets that verify the hypotheses of \cite[Theorem 1.2]{LP09}. This leads to a natural question whether the technical growth conditions (b) on $C_1, C_2$ are truly necessary, and whether progressions exist in any set of large enough Fourier dimension. This naive expectation is however false. Shmerkin \cite[Theorems A and B]{S17} has recently proved the existence of  a compact full-dimensional Salem set  contained in $[0,1]$ that avoids all nontrivial arithmetic progressions. The existence of such a Salem set seems, at first glance, to contradict the conventional belief that such sets should enjoy richer structure. 
\subsubsection{Rational linear patterns} The main content of our first three results is that even though a Salem set of large dimension can avoid a specific linear pattern (or even finitely many) given by $\mathcal F$, it cannot avoid all of them.   

\begin{theorem}\label{t-2}
Given $v\in\nn$, $v\ge2$, let $E\subseteq [0,1]$ be a closed set satisfying $\dim_\ff (E)>\frac2{v+1}$; i.e.,  there exist some $\bz>\frac1{v+1}$, a probalility
measure $\mu$ supported on $E$
and some positive constant $C$ such that
\begin{equation}\label{star}
|\hat \mu(\xi)|\le C(1+|\xi|)^{-\bz}.
\end{equation}
Then $E$ contains a nontrivial zero of some $f \in \cf_v(\mathbb N)$ defined in \eqref{def-Fn}. In other words, there exists $\{m_0,\ldots,m_v\} \st \mathbb N$
satisfying $m_0=\sum_{i=1}^vm_i$,
such that $E$ contains a nontrivial solution of the 
equation 
\begin{equation}\label{ab}
\sum_{i=1}^vm_ix_i=m_0x_0.
\end{equation}
\end{theorem}
\begin{corollary} 
Let $E \subseteq [0,1]$ be a closed set of positive Fourier dimension. Then $E$ contains a nontrivial translation-invariant rational linear pattern, in the sense of Definition \ref{definition-pattern}. 
\end{corollary} 
\vskip0.1in 
\noindent {\em{Remarks: }} \begin{enumerate}[1.]
\item We compare Theorem \ref{t-2} with earlier results of K\"orner \cite{{Korner1}, {Korner2}}. For instance, in \cite[Lemma 2.3]{Korner2} he shows that if $E$ is a subset of the unit circle $\mathbb T = \mathbb R/\mathbb Z$ with $\dim_{\mathbb F}(E) > \frac{2}{(v+1)}$, then there exist integers $m_0, m_1, \cdots, m_v \in \mathbb Z$, not all zero, and distinct points $x_0, x_1, \cdots, x_v \in E$ such that 
\begin{equation} m_0 x_0 = m_1 x_1 + \cdots + m_v x_v \text{ (mod 1)}. \label{Korner} \end{equation}  Apriori, one does not know the number of integers $m_j$ that are zero in the above equation, the signs of the nonzero integers $m_j$ and whether the equation is translation-invariant. On the other hand, the linear equations stemming from $\cf_v(\mathbb N)$ and underlying Theorem \ref{t-2} are exact (not modulo integers), and the coefficients $m_0, \cdots, m_v$ are all positive with the further constraint $m_0 = m_1 + \cdots + m_v$. K\"orner \cite[Theorem 2.4]{Korner2} also constructs a set $E \subseteq \mathbb T$ of Fourier dimension $1/v$ with the following property: there does not exist any nonzero vector $(m_0, \ldots, m_v) \in \mathbb Z^{v+1}$ for which the equation \eqref{Korner} admits a nontrivial solution consisting of distinct points $x_0, x_1, \ldots, x_v \in E$. K\"orner's construction is based on a Baire category argument,  and therefore non-explicit. We ask the interested reader to compare K\"orner's construction of an avoiding set with the avoidance results in this paper (Theorems \ref{t-1}, \ref{t-3} and \ref{t-5}), which avoid more restricted classes of equations but are of larger Fourier dimension. 
\vskip0.1in
\item As another point of contrast, we mention a construction of Keleti \cite{K08} that provides, for any countable set $T \subseteq (0,1)$, a subset $E \subseteq [0,1]$ of Hausdorff dimension 1 that does not contain any triple of distinct points $\{x, y, z\}$ such that $tx + (1-t)y = z$ for any $t \in T$. Choosing $v = 2$ and $T = \mathbb Q \cap (0,1)$, the set of rationals in $(0,1)$, we observe that $\dim_{\mathbb F}$ in Theorem \ref{t-2} cannot be replaced by $\dim_{\mathbb H}$. Generalizing Keleti's result, Math\'e \cite{M17} proves the existence of a rationally independent set in $\mathbb R$ of full Hausdorff dimension.  We recall that a set $E \subseteq \mathbb R$ is {\em{rationally independent}} if for any integer $v \geq 2$ and any choice of distinct points $x_1, x_2, \cdots, x_v \in E$,
\[ \sum_{j=1}^{v} a_j x_j = 0 \text{ with } \{a_1, \cdots, a_v \} \subseteq \mathbb Z \qquad \text{ implies } \qquad a_1 = a_2 = \cdots = a_v = 0.  \] 
  Theorem \ref{t-2} implies that such sets cannot be Salem. Indeed any set $E \subseteq \mathbb R$ of positive Fourier dimension will support a probability measure $\mu$ that satisfies \eqref{star} for some $v \in \mathbb N$ and some $\beta > 1/(v+1)$. By Theorem \ref{t-2} it will contain a rationally dependent $(v+1)$-tuple of distinct points that obeys a relation of the form \eqref{ab}. 
\end{enumerate} 
\begin{corollary} \label{Cor-noSalem} 
There can be no rationally independent set in $\mathbb R$ of positive Fourier dimension.   
\end{corollary} 
%
%
\noindent However, it is possible for a large Salem set to avoid nontrivial zeros of any {\em{finite}} sub-collection of $\cf$, as our next result illustrates.  
\begin{theorem}\label{t-1}
Let $\cf$ be as in \eqref{def-cF}. Given any finite collection $\cg \st\cf$,
there exists a set $E\st [0,1]$ with $\dim_\ff E=1$
such that $E$
contains no nontrivial zero of any $f\in\cg$.
\end{theorem}
\noindent  Corollary \ref{Cor-noSalem} and Theorem \ref{t-1} lead to a natural question: {\em{does there exist a full-dimensional Salem set that manages to avoid the nontrivial zeros of some countably infinite sub-collection of $\cf$?}} We answer this question in the affirmative; see Theorem \ref{t-5}. At the moment, we do not know how to characterize such sub-collections.   
\vskip0.1in
\noindent Our next result attempts to strike a balance of a different sort between Theorems \ref{t-2} and \ref{t-1}. While Theorem \ref{t-2} dictates that a Salem set of large Fourier dimension must necessarily contain a nontrivial zero $(x_0, x_1, \ldots,  x_v)$ of some function $f \in \mathcal F_v$ and some $v \geq 2$, it apriori does not specify the diameter or spread of such a solution, 
\[ \text{diam}(x_0, \ldots, x_v) = \max \bigl\{|x_i - x_j| ; i, j \in \{0, 1, \cdots, v\}, i \ne j \bigr\}, \] 
which could in principle be very small; in other words the nontrivial solution could be ``almost trivial". We now show that it is possible to construct a full-dimensional Salem set that prohibits, in a quantifiable way, nontrivial zeros from being almost trivial.    
\begin{theorem} \label{t-3} 
There exists a set $E \subseteq [0,1]$, $\dim_\ff E=1$ with the following property. For every $v \geq 2$ and every $f \in \mathcal F_v(\mathbb N)$ defined as in \eqref{def-Fn}, there exists $\kappa > 0$ such that whenever there exists a $(v+1)$-tuple $(x_0, x_1, \cdots, x_v) \in E^{v + 1}$ with \begin{equation}  \text{diam}(x_0, x_1, \cdots, x_v) < \kappa \quad  \text{ and  } \quad f(x_0, x_1, \cdots, x_v) = 0, \label{small-diam}  \end{equation}  
we have that $x_0 = x_1 = \cdots = x_v$.  In particular, a nontrivial zero of $f$ in $E$, if it exists, would obey diam$(x_0, \cdots, x_v) \geq \kappa$. 
\vskip0.1in
\noindent In addition, the constant $\kappa = \kappa_N$ can be chosen uniformly for all $f \in \cf$ whose coefficients are bounded by $N$.  
\end{theorem} 

\subsubsection{General linear patterns}\label{irrational}
\noindent The statements of Theorems \ref{t-2} and \ref{t-1} lead to an interesting possibility. Let $\cf_v(\mathbb R_{+})$ denote the class of translation-invariant linear functions in $(v+1)$ variables with real positive coefficients. Then $\cf_v(\mathbb R_{+})$ can be identified with the $(v-1)$-dimensional set 
\begin{equation}  \ct_{v} = \{ \mathbf t \in (0,1)^{v-1}: t_1 + t_2 + \cdots + t_{v-1} < 1\}, \label{def-T}  \end{equation}   which is a half-space of $\mathbb R^{v-1}$ restricted to the open unit cube, via the map
\begin{align*}  \mathbf t &= (t_1, \cdots, t_{v-1}) \in \ct \longmapsto f_{\mathbf t} \in \cf_v(\mathbb R_{+}), \text{ where } \\ f_{\mathbf t}(x) &=  x_0 - (t_1 x_1 + t_2 x_2 \cdots + t_v x_v), \qquad t_v = 1 - \sum_{i=1}^{v-1} t_i. \end{align*} 
Under this map, the class $\cf_v(\mathbb N)$ is identified with the positive rationals in $\ct_{v}$, and hence is of Hausdorff dimension zero. On the other hand, $\cf_v(\mathbb R_{+})$ is of positive $(v-1)$-dimensional Lebesgue measure. One is then led to ask: {\em{Given a collection $\overline{\cf} \subseteq \cf_v(\mathbb R_{+})$ that is of positive Lebesgue measure or large Hausdorff dimension under this identification, does there exist a set $E \subseteq \mathbb R$ of large Fourier dimension that avoids all nontrivial zeros of $\overline{\cf}$?}}
In our next two theorems, we answer this question in the affirmative, in the special case of trivariate equations, where $v=2$ and $\overline{\cf}$ can be viewed as a subset of $(0,1)$.  In Theorem \ref{t-4}, the class $\overline{\cf}$ is identified with a union of intervals, in Theorem \ref{t-6} with collections of badly approximable numbers.  

\begin{theorem}\label{t-4}
Let us fix any $p\in\nn$ with $p\ge2$ and any $\az\in(0,1)$.
Then there exists some $\kz=\kz(p,\az) > 0$ and $E\subseteq [0,1]$ with $\dim_\ff (E)\ge\az$
such that $E$
contains no nontrivial solution of 
$$tx+(1-t)y=z\quad\mbox{ for all }
\quad t\in\bigcup_{q=1}^{p-1}\Biggl(\frac qp-\kz,\frac qp+\kz\Biggr).$$
\end{theorem}
\vskip0.1in 
\noindent For fixed constants $0 < \tau, c \leq 1$, let us define the collection $\ce_{c, \tau}$ of badly approximable numbers as follows, 
\begin{equation} \label{badly-approximable} 
\ce_{c, \tau}:=\Bigl\{t\in(0,1):\
\Bigl|t-\frac qp\Bigr|>\frac c{p^{1+\tau}},\
\mbox{ for all }\frac qp\in\qq, p \in \nn, q \in \mathbb Z, \text{gcd}(p, q) = 1\Bigr\}.
\end{equation} 
Sets of this type have applications in number theory, and their sizes have been widely studied. For example if $\tau = 1$, then the Hausdorff dimension of $\ce_{c, \tau}$ is of the order of $1 - O_c(1)$ as $c \rightarrow 0$. We refer the reader to \cite[Theorem 1.3]{Simmons-2018} and the bibliography in this article for a survey of such results.  
\begin{theorem}\label{t-6}
For every $\eps_0 \in(0,\frac12)$, there exists a set $E\subseteq [0,1]$ with $\dim_\ff (E)=\frac 1{1+\tau}$
such that $E$
contains no nontrivial solution of 
$$tx+(1-t)y=z, \quad {\text{ for any }} t\in\ce_{c, \tau} \cap(\eps_0,1- \eps_0). $$
\end{theorem} 
\noindent The combined strategies of Theorems \ref{t-4} and \ref{t-6} imply the following corollary.
\begin{corollary}\label{t-7}
Let us fix $0 < \tau, c \leq 1$, $\eps_0\in(0,\frac12)$ and $p\in\nn \setminus \{1\}$.
Then for all sufficiently large $M\in\nn$ and $\kz=\frac1{2pM}$,
there exists $E\st [0,1]$ with $\dim_\ff E=\frac 1{1+\tau}$
such that $E$
contains no nontrivial solution of 
$$tx+(1-t)y=z, \quad \text{ for any } t\in \bigl[\cec\cap(\eps_0,1-\eps_0) \bigr]\cup \Biggl[\bigcup_{q=1}^{p-1} \Bigl(\frac qp-\kz,\frac qp+\kz \Bigr) \Biggr]. $$
\end{corollary}
\noindent The sets of forbidden coefficients $t$ in Theorems \ref{t-4} and \ref{t-6} are large, as a consequence of which the avoiding sets we obtain are not of full dimension. {\em{Is it possible to construct a full-dimensional Salem set for which the set of forbidden coefficients is still quantifiably large?}} Our next result provides an affirmative answer to this question, while also addressing the question posed after Theorem \ref{t-1}. 
\begin{theorem}\label{t-5}
There exists an infinite set $\mathfrak C \subseteq (0,1)$ 
and $E\st [0,1]$ with $\dim_\ff E=1$
such that $E$
contains no nontrivial solution of 
\begin{equation} t x+(1-t)y=z \quad \text{ for any } t \in \mathfrak C. \label{t0-eqn} \end{equation} 
The set $\mathfrak C$ contains infinitely many rationals and uncountably many irrationals. 
\end{theorem}
\noindent It is natural to ask whether there exists a version of Shmerkin's theorem \cite{S17} or Theorem \ref{t-1} for a finite but arbitrary collection of equations in $\cf_v(\mathbb R_{+})$; for instance, does there exist a full-dimensional Salem set $E$ that contains no nontrivial solution of $tx+(1-t)y=z$, for any pre-specified irrational $t\in(0,1)$? We are currently unable to provide an answer to this question. Also, the proof techniques of this paper are not immediately generalizable to other types of translation invariant equations, for example when \[ \sum_{i=1}^{m}s_i x_i  = \sum_{j=1}^{n}t_j y_j  \text{ with } \sum_{i=1}^{m}s_i  = \sum_{j=1}^{n} t_j=1, \quad 0 < s_i, t_j < 1 \text{ and } m, n \geq 2, \] 
or when the equation is nonlinear, say $x_3 - x_1 = (x_2- x_1)^2$. We hope to pursue these directions in future work.      

\subsection{Proof overview and layout of the article} 
Predictably, the proof of our ``non-avoidance'' result Theorem \ref{t-2} is very different from its ``avoidance" counterparts (Theorems \ref{t-1}, \ref{t-4}, \ref{t-6} and \ref{t-5}). We start with the former. Let us recall from \eqref{def-T} the definition of the set $\ct_{v}$, and set 
\begin{equation} \ct^{\ast}_v := \Bigl\{ (t_1, \ldots, t_v) \in  (0,1)^v:  t_1 + t_2 + \cdots + t_v =1 \Bigr\}. \label{def-Tstar} \end{equation}
Thus $\ct_{v}$ is the projection of $\ct_{\nu}^{\ast}$ onto the first $(v-1)$ coordinates.  
Given a probability measure $\mu$ supported on $E$ with strong enough Fourier decay, and any vector $\mathbf t = (t_1, \ldots, t_v) \in \ct^{\ast}_v$,
 we construct in Section \ref{measure-pattern-section} a measure $\Lambda_{\mathbf t}$, which, if nontrivial, would signal existence of nontrivial zeros of the function $f_{\mathbf t} (x_0, \cdots, x_v) = x_0 - (t_1 x_1 + \cdots + t_v x_v)$ in $E$. The precise statement may be found in Proposition \ref{p-2}. The construction of the measure $\Lambda_{\mathbf t}$ and verification of its support properties form the main content of this section. The family of measures $\{ \Lambda_{\mathbf t} : \mathbf t \in \ct^{\ast}_v \}$ is then used in Section \ref{positive-proof-section} to create an auxiliary function $F$ on $\ct^{\ast}_v$. On one hand, $F$ will be continuous on $\ct^{\ast}_v$. On the other, it will vanish at every point $\mathbf t \in \ct_{v}^{\ast}$ with rational entries, provided $E$ avoids all nontrivial zeros of $\cf$. This would then force $F$ to be identically zero on its domain. As we will see in Section \ref{proof-t-2}, this leads to a contradiction to the assumption that $\mu$ has nontrivial mass. The proof of Theorem \ref{t-2} appears here.   
\vskip0.1in
\noindent There are two basic themes underpinning Theorems \ref{t-1}, \ref{t-3}, \ref{t-4}, \ref{t-6}, \ref{t-5} and Corollary \ref{t-7}. Each of their proofs consists of two parts. The first step involves constructing an avoiding set in the integers that can be easily transferred to the continuum. The resulting set in the continuum is a disjoint union of intervals; as such, it  cannot completely avoid nontrivial zeros of translation-invariant linear functions. However, it inherits a partial avoidance feature from its discrete counterpart, in the sense that points from distinct intervals cannot form a nontrivial zero. This partial avoidance feature is later replicated on many scales to achieve full avoidance. 
\vskip0.1in
\noindent Constructions of subsets of integers that have large cardinality and avoid linear patterns abound in the combinatorial and number-theoretic literature, and constitute an independent research direction in its own right. A classical construction of this type is due to Behrend \cite{B46}, who obtained a large progression-free set in the integers. The seminal work of Ruzsa \cite{{Ruzsa1},{Ruzsa2}} has also led to numerous constructions and applications in the discrete setting. Refinements of Behrend's construction play an important role in our paper, specifically in the proofs of Theorems \ref{t-1}, \ref{t-3}, \ref{t-4} and \ref{t-5}. The precise statement concerning the existence of a Behrend-like set appears in Proposition \ref{p-1}. With slight variations, this construction is then lifted to the continuum in Lemmas \ref{l-1} and \ref{l-4} and in Proposition \ref{C-construction-prop}. The sets produced by this latter group of results are the unions of intervals that partially inherit the avoidance feature of their Behrend-type parents. They serve as building blocks for sets in the continuum that fully avoid certain linear patterns. For instance, Lemma \ref{l-1} leads to Theorems \ref{t-1} and \ref{t-3}, while Lemma \ref{l-4} and Proposition \ref{C-construction-prop} imply Theorems \ref{t-4} and \ref{t-5} respectively. The discrete avoiding sets that underlie Theorem \ref{t-6} and Corollary \ref{t-7}, on the other hand, use number-theoretic properties of $\mathcal E_{c, \tau}$, the forbidden class of coefficients consisting of badly approximable numbers. These building blocks appear in Lemmas \ref{l-5} and \ref{l-6} respectively. 
\vskip0.1in
\noindent The discrete constructions described above ensure avoidance of certain linear patterns in their continuum analogues, but apriori ensure no Fourier dimensionality, which is a critical requirement in our work. The second fundamental theme in our avoidance theorems is a more recent idea of Shmerkin \cite{S17} that embeds random translates of copies of Behrend-like sets in a Cantor-type construction. As we see in \cite[Theorem 2.1]{S17} (quoted in this article in Theorem \ref{t-A}), these random translations ensure optimal Fourier decay of the resulting Cantor measure subject to dimension, under very mild restrictions on the construction parameters. While there are many instances in the literature where randomization induces optimal Fourier decay \cite{{Bluhm-1}, {Bluhm-2}, {Kahane-1}, {Kahane-2}, {LP09}, {Salem}, {S17}}, the specific use of random translates retains the avoidance property of the set regarding the linear equations, which themselves are translation-invariant. As we demonstrate in the proofs of our theorems, these two strategies are flexible, robust, interact well with each other, and can be adapted in a number of ways to generate algorithms for many avoiding sets of large Fourier dimension.     

\subsection{Acknowledgements} This work was initiated in the academic year 2017-2018, when YL was visiting University of British Columbia on a study leave from Beijing Jiaotong University, funded by a scholarship from China Scholarship Council. He would like to thank all three organizations for their support that enabled his visit. YL was also supported by National Natural Science Foundation of China (Grant nos.~11601028, 11771446 and 1971402) and the Fundamental Research Funds for the Central Universities of China (Grant nos.~2019RC014). MP was partially supported by a 2018 Wall Scholarship from the Peter Wall Institute of Advanced Study, a 2019 Simons Fellowship and two NSERC Discovery grants. 


\section{Identification of linear patterns via measures} \label{measure-pattern-section}
\subsection{A measure on the set of solutions of a linear equation} 
For $\ct^{\ast}_v$ as in \eqref{def-Tstar}, let $\mathbf t = (t_1, \ldots, t_v) \in \ct^{\ast}_v$. We consider the translation-invariant linear equation 
\begin{equation}  \label{n+1-linear-equation} x_{v+1} = t_1 x_1 + \cdots + t_v x_v. \end{equation}   
Given a closed set $E \subseteq [0,1]$, our goal in this section is to construct a measure $\Lambda_{\mathbf t}$ which, if nontrivial, would imply the existence of nontrivial solutions of \eqref{n+1-linear-equation} in $E$. We follow an idea initially introduced in \cite{LP09}, and explored further in \cite{{CLP}, {HLP}}. Unlike previous work, which used measures like $\Lambda_{\mathbf t}$ to establish existence of specific configurations, our final goal is to apply the entire family of measures $\{\Lambda_{\mathbf t} : \mathbf t \in \ct^{\ast}_v \} $ towards a contradiction, in the proof of Theorem \ref{t-2}, establishing in the process that not all these measures can be simultaneously trivial. This proof appears in Section \ref{proof-t-2}.  
 \vskip0.1in
\noindent  Let us define \[ X_{\mathbf t}(E) := \Bigl\{ (x_1, \cdots, x_v) \in E^v : \sum_{i=1}^{v} t_i x_i \in E \Bigr\} \subseteq E^v. \] The set $X_{\mathbf t}(E)$ is always nonempty, since diagonal vectors of the form $(x, \cdots, x) \in E^v$ always lie in $X_{\mathbf t}(E)$. However, for an arbitrary Borel set $E$, the set $X_{\mathbf t}(E)$ need not contain any nontrivial solution of \eqref{n+1-linear-equation} in general. Nonetheless, with some additional assumptions on $E$, it is possible to endow $X_{\mathbf t}(E)$ with a certain non-negative Borel measure. Apriori, this measure could be trivial, i.e., it is possible to have $\Lambda_{\mathbf t}(X_{\mathbf t}(E)) = 0$. However, if nontrivial, this measure assigns zero mass to the trivial solutions of \eqref{n+1-linear-equation}. The construction of this measure is the main content of Proposition \ref{p-2} below, which is a main ingredient in the proof of Theorem \ref{t-2}. 
\vskip0.1in
\noindent Let us fix a function $\psi\in C_{c}^{\infty}(\mathbb R)$ with $\psi\geq 0$, supp$(\psi) \subseteq [-1,1]$ and $\int \psi=1$. Let $\mu$ be a probability measure supported on $E$. Set 
\[ \psi_\varepsilon(x):=\frac{1}{\eps}  \psi\left(\frac{x}{\varepsilon}\right) \quad \text{ and } \quad  \mu_\eps:=\mu*\psi_\varepsilon \text{ for } \varepsilon > 0. \] Then the measure $\mu_{\eps}$ is a smooth probability density  supported on $\mathcal N_{\eps}[E]$, the $\eps$-neighbourhood of $E$. With these definitions in place, we introduce a measure $\Lambda_{\mathbf t}^{(\eps)}$ for every $\eps > 0$ and every $\mathbf t \in (0,1)^v$: 
\begin{equation} \label{nute-def}
\la\Lambda_{\mathbf t}^{(\varepsilon)}, \mathtt f\ra:= \int_{[0,1]^{v}} \mathtt f(x_1,\cdots,x_v) \; \mu_\eps\Bigl(\sum_{i=1}^v t_ix_i \Bigr)\, \prod_{i=1}^{v} \mu_{\varepsilon}(x_i) \; dx_1 \cdots dx_v, \quad \mathtt f \in C([0,1]^v).
\end{equation} 
For each fixed $\eps > 0$, the measure  $\Lambda_{\mathbf t}^{(\eps)}$ is clearly non-negative, absolutely continuous with respect to the Lebesgue measure on $[0,1]^v$, and supported on the set $X_{\mathbf t}(\mathcal N_{\eps}[E])$. Using these, we  
define a linear functional $\Lambda_{\mathbf t}$ that acts apriori on the vector space $\mathcal V$, consisting of Schwartz functions $\mathtt f: \mathbb R^{v}\rightarrow\rr$ whose Fourier transform $\widehat{\mathtt f}: \mathbb R^v \rightarrow \mathbb C$ is smooth of compact support in $\mathbb R^v$: 
\begin{equation}\label{2.7}
\la\Lambda_{\mathbf t}, \mathtt f\ra:=\lim_{\eps\to 0} \la\Lambda_{\mathbf t}^{(\varepsilon)}, \mathtt f\ra. 
\end{equation}
Our main task is to show that under certain assumptions on the underlying measure $\mu$, the limit in \eqref{2.7} exists, that the functional $\Lambda_{\mathbf t}$ above is well-defined and can be identified with a non-negative finite measure on $X_{\mathbf t}(E)$. 
\begin{proposition}\label{p-2}
Let $v \in\nn$, $v \ge2$. Suppose that $\mu$ is a probability measure supported on a closed set $E \subseteq [0,1]$ satisfying the Fourier decay condition \eqref{star} 
for some $\bz>\frac1{v+1}$ and some positive constant $C$. 
\begin{enumerate}[(a)]
\item \label{nut-parta} If $\mathtt f \in \mathcal V$, then for every $\mathbf t \in (0,1)^v$ the limit in (\ref{2.7}) exists and equals
\begin{align}\label{fab}
\la\Lambda_{\mathbf t}, \mathtt f\ra = \int_{\rr^{v+1}}\widehat\mu(\xi)\widehat\mu(\eta_1) \cdots\widehat\mu(\eta_v)
\widehat{\mathtt f} \left(-\eta_1-t_1\xi,\cdots, -\eta_v-t_v\xi\right)
\,d\xi \,d\eta_1\cdots \,d\eta_v. 
\end{align}
The integral in \eqref{fab} is absolutely convergent. 
\vskip0.1in
\item \noindent  \label{nut-partb} For every $\mathbf t \in (0,1)^v$, there exists a positive, finite constant $C_0(\mathbf t)$ depending only on $C$ and $\mathbf t$ such that for all $\mathtt f \in \mathcal V$, 
\begin{equation}\label{2.8}
\sup_{\eps > 0} \bigl |\la\Lambda_{\mathbf t}^{(\eps)}, \mathtt f\ra \bigr|\leq C_0(\mathbf t) \, \|\mathtt f\|_\infty \quad \text{ and hence } \quad |\la\Lambda_{\mathbf t}, \mathtt f\ra|\leq C_0(\mathbf t) \,  \| \mathtt f\|_\infty. 
\end{equation}
In particular, if $\mathbf t = (s, t, \cdots, t) \in (0,1)^v$ with $s + (v-1)  t = 1$, then the function 
\begin{equation} s \longmapsto C_0 \Bigl( s, \frac{1-s}{v-1}, \cdots, \frac{1-s}{v-1} \Bigr)  \text{ is integrable on $[0,1]$.} \label{t-diagonal} \end{equation} 
\item \label{nut-partc} There exists a non-negative, bounded linear functional $\Lambda_{\mathbf t}$ on $C([0,1]^v)$ such that 
$\Lambda_{\mathbf t}^{(\eps)} \rightarrow \Lambda_{\mathbf t}$ weakly as $\eps \rightarrow 0$, for every $\mathbf t \in (0,1)^v$. In other words, 
\begin{equation} \label{def-nu}   \la\Lambda_{\mathbf t}, \mathtt f\ra:= \lim_{\eps \rightarrow 0} \big\langle \Lambda_{\mathbf t}^{(\eps)}, \mathtt f \big\rangle \text{ exists for every $\mathtt f \in C([0,1]^v)$.} \end{equation}  
The estimate \eqref{2.8} continues to hold for all $\mathtt f \in C([0,1]^v)$. The definition \eqref{def-nu} agrees with \eqref{2.7} and \eqref{fab} for $\mathtt f \in C([0,1]^v) \cap \mathcal V$.
\item \label{nut-partd} The functional $\Lambda_{\mathbf t}$ introduced in part \eqref{nut-partc} is given by integration against a non-negative finite Radon measure, also denoted $\Lambda_{\mathbf t}$. With this identification and assuming that $\Lambda_{\mathbf t} \not\equiv 0$, the following conclusions hold: 
\begin{equation}\label{2.2}
\supp(\Lambda_{\mathbf t}) \subset X_{\mathbf t}(E) :=\Bigl\{(x_1,\ldots,x_v)\in E^v :\sum_{i=1}^v t_ix_i\in E\Bigr\},
\end{equation}
and  for any two indices $i,j\in\{1,\ldots,v\}$ with $i\neq j$,
\begin{equation}\label{2.3}
\Lambda_{\mathbf t} \left(\left\{(x_1,\ldots,x_v)\in[0,1]^{v}:\ x_i=x_j\right\}\right)=0.
\end{equation}
\end{enumerate} 
\end{proposition}
\subsection{The measure $\Lambda_{\mathbf t}$} 
\subsubsection{Proof of Proposition \ref{p-2} \eqref{nut-parta}}
\begin{proof}
Suppose that $\mathtt f \in \mathcal V$. Then, for every $\mathbf t \in (0,1)^v$ and every $\eps > 0$, the defining integral for $\langle \Lambda_{\mathbf t}^{(\eps)}, \mathtt f \rangle$ in \eqref{nute-def} is absolutely convergent. Hence  we can use Fourier inversion to write 
\begin{align}\label{2.9}
\la\Lambda_{\mathbf t}^{(\eps)}, \mathtt f\ra &= \int_{\rr ^{v}} \mathtt f(x_1,\cdots,x_v)
\Biggl[\int_{\rr^{v+1}}
\widehat{\mu}_\eps(\xi)\widehat\mu_{\eps}(\eta_1)\cdots \widehat \mu_{\eps}(\eta_{v}) \notag \\ 
&\hskip0.5in \times \exp \Biggl( 2\pi i \Bigl(\xi\cdot \sum_{i=1}^v t_ix_i+\sum_{i=1}^v x_i\eta_i \Bigr) \Biggr) \,d\xi\,d\eta_1 \cdots\,d\eta_v \Bigg]
\,dx_1\cdots\,dx_v
\notag\\
&= \int_{\rr^{v+1}}
\widehat{\mu}_\eps(\xi)\widehat\mu_{\eps}(\eta_1) \cdots\widehat\mu_{\eps}(\eta_v)
\Biggl [\int_{\rr^{\nu}}\mathtt f(x_1,\cdots, x_v) \notag
\\&\hskip0.5in \times 
\exp \Biggl( 2\pi i \sum_{i=1}^v x_i \left(\eta_i+t_i\xi \right) \Biggr)
\,dx_1\cdots\,dx_v\Bigg]\,d\xi \,d\eta_1 \cdots 
\,d\eta_v \notag\\
&= \int_{\rr^{v+1}}\widehat{\mu}_\eps(\xi)\widehat \mu_{\eps}(\eta_1) 
\cdots\widehat\mu_{\eps}(\eta_v)
\widehat{\mathtt f} \left(-\eta_1-t_1\xi,\cdots, -\eta_v -
t_v\xi\right)\,d\xi \,d\eta_1\cdots \,d\eta_v.
\end{align}
We observe that 
\begin{equation}
\widehat{\mu}_\eps(\xi)=\widehat\mu(\xi)\widehat{\psi}_\varepsilon(\xi)=
\widehat\mu(\xi)\widehat\psi(\varepsilon\xi),\mbox{ for all }\xi\in\rr.
\end{equation}
Since $\int\psi=1$, we have $\widehat\psi(0)=1$. Hence 
$\widehat{\mu}_\eps(\xi)\rightarrow\mu(\xi)\,\, \mbox{as}\,\,\eps \rightarrow 0$ for all $\xi\in\rr$, which means that the integrand in \eqref{2.9} converges pointwise to the integrand in \eqref{fab}.  Moreover, $|\widehat\psi(\xi)|\leq 1$ for all $\xi\in\rr$, and therefore  $| \widehat{\mu}_\eps(\xi) | = | \widehat{\mu}(\xi) \widehat{\psi}(\eps \xi) | \leq | \widehat{\mu}(\xi) |$.  Combining the two observations above, the desired representation \eqref{fab} will follow from the dominated convergence theorem, provided  we show that 
\begin{align}\label{2.10}
\mathscr{I}:=\int_{\rr^{v+1}}\lf|\,\widehat\mu(\xi)\widehat\mu(\eta_1) \cdots\widehat\mu(\eta_v)
\widehat{\mathtt f} \left(-\eta_1-t_1\xi,\cdots, -\eta_v-t_v\xi\right)
\right|\,d\xi \,d\eta_1\cdots \,d\eta_v<\infty.
\end{align} 
 \vskip0.1in
 \noindent We set about proving the convergence of this integral. Since $\mathtt f \in \mathcal V$, there exists a constant $R > 0$ such that $\supp(\widehat{\mathtt f}) \subseteq [-R, R]^v$. The domain of integration for the integral in (\ref{2.10}) is therefore 
$$\Omega:=\left\{(\xi,\eta_1, \cdots,\eta_v)\in\rr^{v+1}:\,\left|-\eta_i-t_i\xi\right|\leq R, \; 1 \leq i \leq v \right\}.$$
For fixed $\xi\in\rr$, we set 
\begin{align} \mathscr{J}(\xi; \mathbf t) &:=\int_{\Omega_\xi}\left|\widehat{\mu}(\eta_1) \cdots \widehat{\mu}(\eta_v)
\widehat{\mathtt f} \left(-\eta_1-t_1\xi,\cdots, -\eta_v-t_v\xi\right)\right|
\,d\eta_1\cdots \,d\eta_v, \text{ where } \notag \\ 
\Omega_\xi &:=\bigl\{(\eta_1, \cdots,\eta_v )\in\rr^{v}:\, (\xi, \eta_1, \cdots, \eta_v) \in \Omega \bigr\} \notag \\ 
&\phantom{:}= \bigl\{(\eta_1, \cdots,\eta_v )\in\rr^{v}:\, \left|-\eta_i-t_i\xi\right|\leq R, \; 1 \leq i \leq v \bigr\} \notag \\  & \phantom{:} = \prod_{i=1}^{v} \bigl[-t_i \xi -R, -t_i \xi + R \bigr], \text{ so that } \notag \\  
\mathscr{I} &\phantom{:} = \int_{\mathbb R} \bigl| \widehat{\mu} (\xi) \bigr| \mathscr{J}(\xi; \mathbf t) \, d\xi.     \label{I-again}
\end{align} 
The Fourier decay condition \eqref{star} gives rise to the estimate 
\begin{align}  
\int_{-t_i \xi - R}^{-t_i \xi + R} | \widehat{\mu}(\eta_i) | d\eta_i &\leq C \int_{-t_i \xi - R}^{-t_i \xi + R} (1 + |\eta_i|)^{-\beta} \, d\eta_i  \notag \\ 
&\leq  
\left\{ 
\begin{aligned} 
&C \int_{-3R}^{3R} d\eta_i &\text{ if } t_i|\xi| \leq 2R, \\ 
& C \int_{-t_i \xi - R}^{-t_i \xi + R} \bigl(1 + t_i |\xi|/2 \bigr)^{-\beta} \, d\eta_i &\text{ if } t_i|\xi| > 2R, 
\end{aligned} 
\right\} \notag \\ 
&\leq C(R, \beta) (1 + t_i |\xi|)^{-\beta}.   \notag 
\end{align} 
This in turn implies that 
\begin{equation} 
\mathscr{J}(\xi; \mathbf t) \leq ||\widehat{\mathtt f}||_{\infty} \int_{\Omega_{\xi}} \prod_{i=1}^{v} \bigl| \widehat{\mu}(\eta_i) \bigr| d\eta_1 \cdots d\eta_v \leq  C(R, \beta) ||\widehat{\mathtt f}||_{\infty} \prod_{i=1}^{v} (1 + t_i |\xi|)^{-\beta}. \label{J-estimate} 
\end{equation}
The constant $C(R, \beta)$ in the inequality above is independent of $\mathbf t$, and depends only on $R$, $\beta$ and the constant $C$ in \eqref{star}. Substituting \eqref{J-estimate} into \eqref{2.10} and \eqref{I-again} , we find that 
\[ \mathscr{I} = \int_{\mathbb R} | \widehat{\mu}(\xi) | \mathscr{J}(\xi; \mathbf t) \, d\xi \leq C(R, \beta) ||\widehat{\mathtt f}||_{\infty} \int_{\mathbb R} (1 + |\xi|)^{-\beta} \prod_{i=1}^{v} (1 + t_i |\xi|)^{-\beta} d\xi \leq C(R, \mathtt f, \beta, \mathbf t) < \infty,  \] 
where the last integral converges since $t_i > 0$ for all $i \in \{1, \ldots, v\}$ and $(v+1) \beta > 1$ by assumption. 
\end{proof}  
\subsubsection{Proof of Proposition \ref{p-2} \eqref{nut-partb}}
 \begin{proof} 
By part \eqref{nut-parta}, the limit in \eqref{2.7} exists for $\mathtt f \in \mathcal V$; hence the second inequality in \eqref{2.8} follows from the first. For $\eps > 0$, the defining integral in \eqref{nute-def} absolutely convergent. Hence it follows from the Fourier inversion formula and the non-negativity of $\mu_{\eps}$ that 
\begin{align}\label{2.8-}
\bigl| \bigl \langle \Lambda_{\mathbf t}^{(\eps)}, \mathtt f \big\rangle  \bigr| &=  \Bigl| \int_{\rr^{v}} \mathtt f(x_1,\cdots,x_v)\mu_\eps\Bigl(\sum_{i=1}^v t_ix_i\Bigr)\, \prod_{i=1}^{v}\mu_{\eps}(x_i) dx_1 \cdots dx_v \Bigr| \notag\\
&\leq ||\mathtt f||_{\infty} \int_{\rr^{v}}\mu_\eps\Bigl(\sum_{i=1}^v t_ix_i\Bigr)\, \prod_{i=1}^{v}\mu_{\eps}(x_i) dx_1 \cdots dx_v \notag \\
&\leq ||\mathtt f||_{\infty} \Bigl| \int_{\rr} {\widehat\mu_{\eps}(\xi)} \prod_{i=1}^{v} \widehat\mu_{\eps}(- t_i \xi)\,d\xi \Bigr| \notag \\
&\leq ||\mathtt f||_{\infty} \int_{\rr} \Bigl| {\widehat\mu (\xi)} \prod_{i=1}^{v} \widehat\mu (- t_i \xi) \Bigr| \,d\xi \notag \\
&\leq C^v ||\mathtt f||_{\infty} \int_{\mathbb R} (1 + |\xi|)^{-\beta} \prod_{i=1}^{v} (1 + t_i |\xi|)^{-\beta} \, d\xi \notag \\ 
&\leq C_0(\mathbf t) ||\mathtt f||_{\infty}, \text{ where } \notag \\ C_0(\mathbf t) &:= C^v \int_{\mathbb R} (1 + |\xi|)^{-\beta} \prod_{i=1}^{v} (1 + t_i |\xi|)^{-\beta} \, d\xi < \infty \text{ because } \beta (v+1) > 1. 
\end{align}
In particular, if $\mathbf t = (s, t, \cdots, t) \in (0,1)^v$ with $s + (v-1)t = 1$, then the constant $C_0(\mathbf t)$ in \eqref{2.8-} reduces to
\[ C_0(\mathbf t) = C^v \int_{\mathbb R} (1 + |\xi|)^{-\beta} (1 + s|\xi|)^{-\beta} (1 + t |\xi|)^{-(v-1) \beta} \, d\xi.  \] 
Without loss of generality, replacing the decay exponent \eqref{star} by a smaller $\beta$ if necessary, we may assume $\beta < 1$. For $1/(v+1) < \beta < 1$, we estimate the above integral as follows, 
\begin{align*}
C_0(\mathbf t) &\leq C^v s^{-\beta} \int_{\rr} (1 + |\xi|)^{-2\beta} (1 + t|\xi|)^{-(v-1) \beta}d \xi \\ &\leq C^v s^{-\beta} \Bigl[  \int_{|\xi| \leq t^{-1}} (1 + |\xi|)^{-2\beta} d \xi + \int_{|\xi| > t^{-1}} t^{-(v-1) \beta}  |\xi|^{- (v+1) \beta} \, d\xi \Bigr] \\
&\leq C(v, \beta) s^{-\beta} \left[ \left\{ \begin{aligned} &1 &\text{ if }  2 \beta > 1, \\ &\log(1/t) &\text{ if } 2 \beta = 1, \\ &t^{2 \beta -1} &\text{ if } 2 \beta < 1\end{aligned} \right\} + t^{2 \beta -1}  \right] \\
&\leq C(v, \beta) \times  
\begin{cases} 
s^{-\beta} (1-s)^{2\beta - 1} &\text{ for } \frac{1}{v+1} < \beta < \frac12, \\
s^{-\beta} \log(1/(1-s)) &\text{ for } \beta = \frac12, \\ 
s^{-\beta} &\text{ for } \frac12 < \beta < 1. 
\end{cases} 
\end{align*} 
The last expression is an integrable function of $s$ on $[0,1]$ for any choice of $\beta \in (\frac{1}{v + 1}, 1)$, as claimed in \eqref{t-diagonal}. 

\end{proof} 

\subsubsection{Proof of Proposition \ref{p-2} \eqref{nut-partc}} 
\begin{proof} 
There are two items to be proved here. The first is to show that the definition \eqref{nute-def} of $\Lambda_{\mathbf t}$ on $\mathcal V$ admits a natural extension to $C([0,1]^v)$. The second is to establish that this extended definition obeys \eqref{def-nu} on $C([0,1]^v)$.
\vskip0.1in
\noindent The main observation in support of our first goal is that the class of functions in $\mathcal V$, with domain restricted to $[0,1]^v$, is dense in $C([0,1]^v)$. This follows from the fact that any $\mathtt f \in C([0,1]^v)$ admits a continuous extension $\mathtt f_0 : \mathbb R^v \rightarrow \mathbb C$ with compact support in $\mathbb R^v$. Let $\varphi$ denote a Schwartz function such that $\widehat{\varphi}(0) = 1$ and $\widehat{\varphi} \in C_{c}^{\infty}(\mathbb R^v)$. Then $\varphi_{\eps} = \eps^{-1} \varphi(\cdot/\eps)$ is an approximation to the identity as $\eps \rightarrow 0$. Thus the sequence of functions $\{\mathtt f_0 \ast \varphi_{\eps} : \eps > 0\} \subseteq \mathcal V$ converges uniformly to $\mathtt f_0$ uniformly on compact sets, and hence uniformly to $\mathtt f$ on $[0,1]^v$, as $\eps \rightarrow 0$. 
\vskip0.1in
\noindent Suppose that $\mathtt f \in C([0,1]^v)$ and that $\{\mathtt f_n : n \geq 1\} \subseteq \mathcal V$ is such that $\mathtt f_n \rightarrow \mathtt f$ uniformly on $[0,1]^v$. Then it follows from the second inequality in \eqref{2.8} that for every $\mathbf t \in (0,1)^v$, 
\[ \bigl| \langle \Lambda_{\mathbf t}^{(\eps)}, \mathtt f_n - \mathtt f_m  \rangle \bigr| \leq C_0(\mathbf t) ||\mathtt f_n - \mathtt f_m||_{\infty} \rightarrow 0 \text{ as } n, m \rightarrow \infty, \] 
i.e., the sequence $\{ \langle \Lambda_{\mathbf t}, \mathtt f_n\rangle : n \geq 1\}$ is Cauchy, and hence admits a limit. We denote this limit by $\langle \Lambda_{\mathbf t}, \mathtt f\rangle$.  If $\{\mathtt g_n : n \geq 1\} \subseteq \mathcal V$ is another approximating sequence for $\mathtt f$, then the second inequality in \eqref{2.8} also shows that $\bigl| \langle \Lambda_{\mathbf t}, \mathtt f_n - \mathtt g_n \rangle \bigr| \leq C_0(\mathbf t) || \mathtt f_n - \mathtt g_n||_{\infty} \rightarrow 0$, i.e., the limit does not depend on the approximating sequence $\mathtt f_n$. It is clear from the definition that $\Lambda_{\mathbf t}$ is a non-negative linear functional. The definition of $\Lambda_{\mathbf t}$ also implies that the inequality \eqref{2.8} holds for all $\mathtt f \in C([0,1]^v)$, not merely $\mathtt f \in \mathcal V$. This is then equivalent to the statement that $\Lambda_{\mathbf t}$ is a bounded linear functional on $C([0,1]^v)$.   
\vskip0.1in 
\noindent Next, we will show that the limit in \eqref{def-nu} holds for any $\mathtt f \in C([0,1]^v)$. Let us fix any $\kappa > 0$, and $\mathtt g \in \mathcal V$ such that $||\mathtt f - \mathtt g||_{L^{\infty}[0,1]} < \frac{\kappa}{4C_0(\mathbf t)}$.  It follows from \eqref{fab} that there exists $\eps_0 > 0$ such that  \[ \bigl| \langle \Lambda_{\mathbf t}^{(\eps)} - \Lambda_{\mathbf t}, \mathtt g \rangle \bigr| < \frac{\kappa}{2} \quad \text{ for all } \eps < \eps_0. \]
Combined with \eqref{2.8}, this gives 
\begin{align*} \bigl| \langle \Lambda_{\mathbf t}^{(\eps)} - \Lambda_{\mathbf t}, \mathtt f \rangle\bigr| &\leq \bigl| \langle \Lambda_{\mathbf t}^{(\eps)} - \Lambda_{\mathbf t}, \mathtt g \rangle \bigr|  +  \bigl| \langle \Lambda_{\mathbf t}^{(\eps)} - \Lambda_{\mathbf t}, \mathtt f - \mathtt g \rangle \bigr| \\ 
&\leq \bigl| \langle \Lambda_{\mathbf t}^{(\eps)} - \Lambda_{\mathbf t}, \mathtt g \rangle \bigr| +  \bigl| \langle \Lambda_{\mathbf t}^{(\eps)}, \mathtt f - \mathtt g \rangle \bigr| +  \bigl| \langle \Lambda_{\mathbf t}, \mathtt f - \mathtt g \rangle \bigr| \\ 
&\leq \frac{\kappa}{2} + 2C_0(\mathbf t) ||\mathtt f- \mathtt g||_{\infty} < \kappa, \quad \text{ for all } \eps <  \eps_0. 
\end{align*}   
This establishes the existence of the limit, and completes the proof of \eqref{nut-partc}.  
\end{proof} 

\subsubsection{Proof of Proposition \ref{p-2} \eqref{nut-partd}} 
\begin{proof} 
By the Riesz representation theorem, the linear functional $\Lambda_{\mathbf t}$ can be  identified with a non-negative, finite Radon measure on $[0,1]^v$. For the remainder of this proof, we will use $\Lambda_{\mathbf t}$ to denote this measure. 
\vskip0.1in 
\noindent {\em{Proof of (\ref{2.2}).}} Let us assume that $X_{\mathbf t}(E) \ne \emptyset$. Since $E$ is closed, $X_{\mathbf t}(E)$ is closed as well. In order to establish the support condition \eqref{2.2}, let us fix any $\mathbf u = (u_1, \cdots, u_v) \in [0,1]^v \setminus X_{\mathbf t}(E)$. We aim to show that there exists a small constant $\delta > 0$ such that 
\begin{align} Q(\mathbf u; \delta) &\subseteq [0,1]^v \setminus X_{\mathbf t}(E) \; \text{ and } \notag \\  \langle \Lambda_{\mathbf t}, \mathtt f \rangle = 0 \text{ for any } & \mathtt f \in C([0,1]^v) \text{ with } \supp(\mathtt f) \subseteq Q(\mathbf u; \delta).   \label{supp-nut}\end{align} 
Here $Q(\mathbf u; \delta)$ denotes the closed cube centred at $\mathbf u$ of sidelength $\delta$.  
\vskip0.1in
\noindent  Since $\mathbf u \notin X_{\mathbf t}(E)$, one of the following two conditions must hold: 
\begin{align} 
&\text{either there exists $1 \leq i \leq v$ such that $u_i \notin E$, } \label{case1} \\  
&\text{or $\mathbf u \in E^v$ but $\sum_{i=1}^v t_i u_i \notin E$.  } \label{case2}
\end{align} 
Let us recall that $E$ is closed by definition; hence, if \eqref{case1} holds, then there exists $\delta > 0$ such that $\text{dist}(u_i, E) > 2\delta$, which implies that $\text{dist}(u_i, \mathcal N_{\eps}[E]) > \delta$ for all $\eps < \delta$.
This means that $[u_i - \delta/2, u_i + \delta/2] \cap \mathcal N_{\eps}[E] = \emptyset$, hence $\mathtt f(x_1, \cdots, x_v) \mu_{\eps}(x_i) \equiv 0$ for any $\mathtt f$ supported on $Q(\mathbf u; \delta)$ and any $\eps <  \delta$. In view of \eqref{nute-def}, we then have that $\langle \Lambda_{\mathbf t}^{(\eps)}, \mathtt f \rangle = 0$ for all $\eps < \delta$. Letting $\eps \rightarrow 0$, the conclusion \eqref{def-nu} then implies that $\langle \Lambda_{\mathbf t}, \mathtt f \rangle = 0$.    
\vskip0.1in
\noindent Suppose now that \eqref{case2} holds. Then there exists $\delta > 0$ such that \[ \text{dist}\Bigl(\sum_{i=1}^v t_i u_i, E \Bigr) > 3\delta \quad \text{ and hence } \quad \text{dist}\Bigl(\sum_{i=1}^v t_i u_i,  \mathcal N_{\eps}[E] \Bigr) > 2\delta \;  \text{ for }  \eps < \delta. \] This means that for $(x_1, \cdots, x_v) \in Q(\mathbf u, \delta)$, one has 
\begin{align*} \text{dist} \Bigl(\sum_{i=1}^v t_i x_i, \mathcal N_{\eps}[E] \Bigr) &\geq \text{dist} \Bigl(\sum_{i=1}^v t_i u_i,  \mathcal N_{\eps}[E] \Bigr) - \sum_{i=1}^v t_i |x_i - u_i|  \\ & > 2 \delta - \sum_{i=1}^v t_i \frac{\delta}{2} >  \delta, \text{ since } t_1 + \cdots + t_v = 1. \end{align*}   The same argument as above now implies that $\mathtt f(x_1, \cdots x_v) \mu_{\eps} \bigl(\sum_{i=1}^v t_i x_i \bigr) \equiv 0$ for any $\mathtt f$ supported on $Q(\mathbf u, \delta)$ and $\eps < \delta$. The desired conclusion \eqref{supp-nut} now follows in the same way as in case 1. 
\vskip0.1in
\noindent {\em{Proof of (\ref{2.3}).}} 
Let $\chi:\,\rr\rightarrow [0,\infty)$ be a non-negative Schwartz function such that $\chi\geq 1$
on $[-1,1]$ and $\supp(\widehat{\chi}) \subset[-R,R]$ for some $R\in(0,\infty)$. Set $\chi_\delta(x) :=\chi(\delta^{-1}x)$. Given any two indices $i, j\in \{1, \cdots, v\}$, $i \ne j$, we will prove that
\begin{equation}\label{2.16}
\begin{aligned} 
&\la\Lambda_{\mathbf t}, \mathtt f^{[\delta]} \ra\to 0 \,\text{ as } \,\delta\to 0, \quad \text { where } \quad \mathtt f_0(x_1, \cdots, x_v) = 
\chi_\delta(x_i-x_j) 
\end{aligned} 
\end{equation}
and $\mathtt f^{[\delta]}$ is the restriction of $\mathtt f_0$ to $[0,1]^v$. Clearly $\mathtt f^{[\delta]}$ is non-negative, and $\mathtt f^{[\delta]} \in C([0,1]^v)$. Since $\Lambda_{\mathbf t}\bigl( \bigl\{\mathbf x \in [0,1]^v: x_i = x_j \bigr\} \bigr) \leq \la\Lambda_{\mathbf t}, \mathtt f^{[\delta]} \ra$, proving \eqref{2.16} implies (\ref{2.3}).
\vskip0.1in
\noindent In order to prove \eqref{2.16}, we proceed as in the proof of Proposition \ref{p-2}(\ref{nut-parta}). For every $\eps > 0$, we can use Fourier inversion to write
\begin{align} 0 \leq \big \langle \Lambda_{\mathbf t}^{(\eps)}, \mathtt f^{[\delta]} \big \rangle &\leq \int_{\mathbb R^v} \chi_{\delta}(x_i- x_j) \mu_{\eps} \Bigl( \sum_{k=1}^{v} t_k x_k \Bigr)  \prod_{k=1}^{v} \mu_{\eps}(x_k) dx_1 \cdots dx_v  \notag  \\ &\leq \Bigl| \int_{\rr^{2}}\widehat{\chi}_\delta(\xi)\widehat{\mu}_\eps(\eta)
\widehat{\mu}_{\eps}\left(-\xi-t_i\eta\right)\widehat\mu_{\eps}\left(\xi-t_j\eta\right)
\prod_{\substack{k=1\\k\neq i, j}}^v \widehat{\mu}_{\eps}\left(-t_k\eta\right)
\,d\xi \,d\eta \Bigr| \notag \\  
&\leq \int_{\rr^{2}}\Bigl|\widehat{\chi}_\delta(\xi)\widehat\mu(\eta)
\widehat\mu\left(-\xi-t_i\eta\right)\widehat\mu\left(\xi-t_j\eta\right)
\prod_{\substack{k=1\\k\neq i, j}}^v\widehat\mu\left(-t_k\eta\right)
\Bigr|\,d\xi \,d\eta \label{tobe-est}
\end{align}
The last integral is independent of $\eps$, and hence provides a bound for $\langle \Lambda_{\mathbf t}, \mathtt f^{[\delta]}\rangle$ as well, in light of Proposition \ref{p-2}\eqref{nut-partc}. Our goal is to show that it is bounded by a constant multiple of $\delta^s$, for some $s > 0$. This of course implies \eqref{2.16}.
\vskip0.1in
\noindent The estimation of the integral in \eqref{tobe-est} proceeds as follows. Since $\widehat{\chi}_\delta(\xi)=\delta \widehat{\chi}(\delta\xi)$
is supported on $[-\frac{R}{\delta},\frac{R}{\delta}]$, we have
\begin{align}\label{2.18}
\bigl| \big \langle \Lambda_{\mathbf t}, \mathtt f^{[\delta]} \big \rangle \bigr|  &\leq C \delta\int_{\rr}\Biggl[ \int_{-R/\delta}^{R/\delta}\Bigl|\widehat\mu\left(-\xi-t_i\eta\right)\widehat\mu\left(\xi-t_j\eta\right) \Bigr|\,d\xi \Biggr] \times \Bigl| \widehat{\mu}(\eta) \prod_{\substack{k=1\\k\neq i,j}}^v \widehat\mu\left(-t_k\eta\right) \Bigr|
 \,d\eta\notag\\
&\leq C_{\mathbf t} \delta \int_{\mathbb R} (1 + |\eta|)^{-(v-1) \beta} \Biggl[ \int_{-R/\delta}^{R/\delta}\bigl|\widehat\mu\left(-\xi-t_i\eta\right)\widehat\mu\left(\xi-t_j\eta\right) \bigr|\,d\xi \Biggr] \, d\eta \notag \\
&\leq C_{\mathbf t} \delta \int_{\mathbb R} (1 + |\eta|)^{-(v-1) \beta} \Biggl[ \int_{-R/\delta}^{R/\delta}\bigl|\widehat\mu\left(-\xi-t_i\eta\right) \bigr|^2 \, d\xi \Biggr]^{\frac{1}{2}}\Biggl[ \bigl|\widehat\mu\left(\xi-t_j\eta\right) \bigr|^2\,d\xi \Biggr]^{\frac{1}{2}} \, d\eta \notag \\
&\leq C_{\mathbf t} \delta \int_{\mathbb R} (1 + |\eta|)^{-(v-1) \beta} \mathtt F_{t_i}(\eta) \mathtt F_{t_j}(\eta) \, d\eta \leq C_{\mathbf t}(\beta) \delta ||\mathtt F_{t_i}||_{q} || \mathtt F_{t_j}||_{q},
\end{align}
where at the last step of \eqref{2.18}, we have applied H\"older's inequality with the exponents $p=(v+1)/(v-1)$ and $q = (v+1)$, so that $1/p + 2/q = 1$. Here we have also used the fact that $(v+1) \beta > 1$, which ensures the integrability of the function $\eta \mapsto (1 + |\eta|)^{-p(v-1)\beta}$ on $\mathbb R$. The function $F_t$ in \eqref{2.18} is defined and estimated as follows: for $t \in (0,1)$, 
\begin{align} 
\bigl[ \mathtt F_t(\eta) \bigr]^2 &:=\int_{-R/\delta}^{R/\delta}\bigl|\widehat\mu\left(-\xi-t\eta\right) \bigr|^2 \,d\xi  \notag \\ 
&\leq C \int_{-R/\delta}^{R/\delta} \bigl(1 + |\xi + t\eta| \bigr)^{-2 \beta} \, d\xi \notag \\ 
&\leq C \times \left\{ 
\begin{aligned} &\int_{-{3R}/{\delta}}^{{3R}/{\delta}} \bigl(1 + |\zeta| \bigr)^{-2 \beta} \, d\zeta &\text{ if } t |\eta| \leq \frac{2R}{\delta}, \\ 
&\int_{-R/\delta}^{R/\delta} \bigl(1 + |t \eta|/2 \bigr)^{-2 \beta} d\xi &\text{ if } t |\eta| > \frac{2R}{\delta} 
\end{aligned}  \right\} \label{F_t-est1} \\ 
&\leq C \times  \left\{ \begin{aligned} &\max \Bigl[1,  \Bigl(\frac{R}{\delta} \Bigr)^{1 - 2\beta} \log(R/\delta) \Bigr] &\text{ if } t |\eta| \leq \frac{2R}{\delta}, \\ 
& \frac{R}{\delta} \bigl(1 + t |\eta| \bigr)^{-2\beta} &\text{ if } t |\eta| > \frac{2R}{\delta} 
\end{aligned} \right\}. \label{F_t-est2}
\end{align} 
In the first estimate in \eqref{F_t-est1}, we have used the change of variable $\xi \mapsto \zeta = \xi + t\eta$, along with the inclusion 
\[ \Bigl\{\zeta = \xi + t\eta : |\xi| \leq \frac{R}{\delta}, \; t|\eta| \leq \frac{2R}{\delta} \Bigr\} \subseteq \Bigl[-\frac{3R}{\delta}, \frac{3R}{\delta} \Bigr]. \] 
The second estimate in \eqref{F_t-est1} follows from $|\xi + t \eta| \geq |t \eta| - R/\delta \geq t|\eta|/2$.  
The final pointwise bound in \eqref{F_t-est2} allows us to estimate the Lebesgue norm of $F_t$: 
\begin{align} ||\mathtt F_t||_q &\leq C_{R,t} \Biggl[ \max \Bigl[ 1, \delta^{2 \beta -1} \log(1/\delta) \Bigr]^{\frac{1}{2}} \delta^{-\frac{1}{q}} + \delta^{-\frac{1}{2}} \Biggl(\int_{t|\eta| > \frac{2R}{\delta}} (1 + t|\eta|)^{-\beta q} \, d\eta \Biggr)^{\frac{1}{q}} \Biggr] \notag \\
&\leq C_{R,t} \Biggl[ \max \Bigl[1, \delta^{2 \beta -1} \log(1/\delta) \Bigr]^{\frac{1}{2}} \delta^{-\frac{1}{q}} + \delta^{- \frac{1}{2}+  \beta - \frac{1}{q}} \Biggr] \label{F_t-norm}
\end{align} 
Substituting \eqref{F_t-norm} into \eqref{2.18} yields  
\begin{align*} \delta ||\mathtt F_{t_i}||_{q} ||\mathtt F_{t_j}||_{q}&\leq C_{R, t_i, t_j} \delta \times \begin{cases} \Bigl[ \delta^{-\frac{2}{q}} + \delta^{-1 + 2(\beta - \frac{1}{q})}\Bigr] &\text{ if } \beta > \frac{1}{2}, \\ \delta^{2 \beta -1 - \frac{2}{q}} \log(1/\delta) &\text{ if } \beta \leq \frac{1}{2},  \end{cases} \\
&\leq C_{R, t_i, t_j} \times \begin{cases} \delta^{1-\frac{2}{q}} + \delta^{2(\beta - \frac{1}{q})}&\text{ if } \beta > \frac{1}{2}, \\ \delta^{2(\beta - \frac{1}{q})} \log(1/\delta) &\text{ if } \beta \leq \frac{1}{2}. \end{cases} 
\end{align*} 
Recalling that $q = (v+1)$ and $(v+1) \beta > 1$, we observe that in both cases, there exists some $s > 0$ such that 
\[\bigl| \big \langle \Lambda_{\mathbf t}, \mathtt f^{[\delta]} \big \rangle \bigr| \leq \delta ||\mathtt F_{t_i}||_{q} ||\mathtt F_{t_j}||_{q} \leq C_{R, t_i, t_j}\delta^s, \] and hence tends to $0$ as $\delta \rightarrow 0$. This completes the proof. 
\end{proof} 

\section{Fourier dimensionality and rational linear patterns} \label{positive-proof-section}

\subsection{Proof of Theorem \ref{t-2}} \label{proof-t-2}
\begin{proof}
We proceed by contradiction. Assume if possible that $E$ avoids all nontrivial
solutions of equations of the form \eqref{ab} for all
choices $(m_0,\ldots,m_v)\in \mathbb N^{v+1}$
with $m_0= m_1 + \cdots + m_v$. This is equivalent to saying that $E$ avoids all nontrivial
solution of equations of the form \eqref{n+1-linear-equation} for all choices of 
$\mathbf t = (t_1, \cdots, t_v) \in\qq_+^{v} \cap \ct_v^{\ast}$, 
where $\qq_+$ denotes the set of positive rational numbers and $\ct_v^{\ast}$ is as in \eqref{def-Tstar}. We will show that if $E$ avoids nontrivial solutions of such equations, then it avoids nontrivial solutions of translation-invariant equations of the form  \eqref{n+1-linear-equation} for {\em{all}} $\mathbf t \in \ct_v^{\ast}$, and that this is not possible.  
\vskip0.1in
\noindent Let us start with a function $\mathtt f \in \mathcal V$. The representation \eqref{fab} in Proposition \ref{p-2} ensures that for any $\mathbf t=(t_1,\ldots,t_v)\in (0,1)^v$  
we can define a function $F = F^{[\mathtt f]}$ on $(0,1)^v$ by 
\begin{equation} \label{def-F}
F^{[\mathtt f]}(\mathbf t):=\la \Lambda_{\mathbf t}, \mathtt f\ra=
\int_{\rr^{v+1}}\widehat\mu(\xi) \Bigl[ \prod_{i=1}^{v} \widehat\mu(\eta_i) \Bigr] 
\widehat {\mathtt f}\left(-\eta_1-t_1\xi,\cdots, -\eta_v-t_v\xi\right)
\,d\xi \,d\eta_1\cdots \,d\eta_v.
\end{equation} 
We will establish momentarily, in Lemma \ref{cty-lemma} below, that $\mathbf t \mapsto F^{[\mathtt f]}(\mathbf t)$ is a continuous function on $(0,1)^v$, and hence on $\ct_v^{\ast}$. Our assumption on $E$ implies that \[ X_{\mathbf t}(E) =\emptyset  \quad \text{ for any } \mathbf t\in\qq_+^{v} \cap \ct_v^{\ast}, \]  
hence we deduce from the support condition \eqref{2.2} that $\Lambda_{\mathbf t} \equiv0$ for such $\mathbf t$. Thus $F^{[\mathtt f]}(\mathbf t)=0$ for every $\mathbf t\in\qq_+^{v} \cap \ct_v^{\ast}$. The continuity of $F^{[\mathtt f]}$ now allows us to conclude that 
\begin{equation} 
F^{[\mathtt f]}(\mathbf t) = 0  \text{ for all } \mathbf t\in \ct_v^{\ast}  \text{ and all }  \mathtt f \in \mathcal V. \label{F=0}
\end{equation} 
 Since $C([0,1]^v) \cap \mathcal V$ is dense in $C([0,1]^v)$ under sup norm, it follows from the boundedness of $\Lambda_{\mathbf t}$ (specifically \eqref{2.8} and part \eqref{nut-partc} in Proposition \ref{p-2}) that 
\begin{equation} \label{will-contradict} 
\langle \Lambda_{\mathbf t}, \mathtt f \rangle = 0 \text{ for all }  \mathtt f \in C([0,1]^v) \text{ and all }  \mathbf t \in \ct_v^{\ast}.  
\end{equation} 
We will use this vanishing property to arrive at the desired contradiction.  
\vskip0.1in
\noindent For every $\eps > 0$, we compute  the mass of $\mu_{\eps}$ in two ways:
\begin{align}
1&= \Biggl[  \int_{0}^{1} \mu_{\eps}(t) \, dt \Biggr]^{v+1} \label{left}  \\ 
&= \int_{[0,1]^{v+1}}
\mu_\eps(x_1)\cdots\mu_\eps(x_v)\mu_\eps\lf(z\r)\,dx_1\cdots\,dx_v\,dz \notag \\
&=\frac1{(v+1)!}
\int_{\{(x_1,\cdots,x_v,z)\in[0,1]^{v+1}:\ x_1<z<x_2<\cdots<x_v\}}
\mu_\eps(x_1)\cdots\mu_\eps(x_v)\mu_\eps\lf(z\r)\,dz\,dx_1\cdots\,dx_v \notag \\
&= \frac1{(v+1)!} \int_{\{(x_1,\cdots,x_v)\in[0,1]^{v}:\ x_1<x_2<\cdots<x_v\}}\int_0^1
\mu_\eps(x_1)\cdots\mu_\eps(x_v) \notag
\\&\hskip0.5in \times
\mu_\eps\Biggl(sx_1+\frac{1-s}{v-1}(x_2+\cdots+x_v)\Biggr)\lf|x_1-\frac1{v-1}(x_2+\cdots+x_v)\r|\,ds\,dx_1\cdots\,dx_v \notag \\
&\le\frac1{(v+1)!} \int_{[0,1]^{v}}\int_0^1
\mu_\eps(x_1)\cdots\mu_\eps(x_v) \notag 
\\&\hskip0.5in \times
\mu_\eps\Biggl(sx_1+\frac{1-s}{v-1}(x_2+\cdots+x_v)\Biggr)\lf|x_1-\frac1{v-1}(x_2+\cdots+x_v)\r|\,ds\,dx_1\cdots\,dx_v \notag \\
&=\frac1{(v+1)!}\int_0^1 \Biggl[ \int_{[0,1]^{v}} \mu_\eps(x_1)\cdots\mu_\eps(x_v) \notag
\\&\hskip0.5in \times  \mu_\eps \Biggl(sx_1+\frac{1-s}{v-1}(x_2+\cdots+x_v)\Biggr)
\lf|x_1-\frac1{v-1}(x_2+\cdots+x_v)\r|\,dx_1\cdots\,dx_v \Biggr] \,ds \notag 
\\ &= \int_{0}^{1} \big\langle \Lambda_{\mathbf t(s)}^{(\eps)}, \mathtt f_{\ast} \big\rangle \, ds.  \label{right} 
\end{align}
In the above sequence of calculations, we have made the change of variable $z \mapsto s$ at the fourth step, with $z = sx_1 + (1-s)(x_2 + \cdots + x_v)/(v-1)$.  The last step follows from the definition \eqref{nute-def}, with
\begin{equation} \label{def-t(s)} \mathbf t(s) = \Biggl(s, \frac{1-s}{v-1}, \cdots, \frac{1-s}{v-1} \Biggr) \; \text{ and } \; \mathtt f_{\ast}(x_1,\cdots,x_v) :=\lf|x_1-\frac1{v-1}(x_2+\cdots+x_v)\r|. \end{equation}  
Clearly $\mathtt f_{\ast} \in C([0,1]^v)$. Combining the left and right ends \eqref{left} and \eqref{right} of the steps above, we arrive at the inequality 
\[\int_{0}^{1} \big\langle \Lambda_{\mathbf t(s)}^{(\eps)}, \mathtt f_{\ast} \big\rangle \, ds \geq 1.  \] 
Now let $\eps \rightarrow 0$. It follows from \eqref{t-diagonal} in Proposition \ref{p-2} \eqref{nut-partb} that for $\mathbf t(s)$ as in \eqref{def-t(s)}, there exists a positive function $C_0(\mathbf t(s))$, integrable on $[0,1]$, such that
\[ \sup_{\eps > 0} \bigl| \langle \Lambda_{\mathbf t}^{(\eps)}, \mathtt f_{\ast} \rangle \bigr| \leq C_0(\mathbf t(s))  \; || \mathtt f_{\ast}||_{\infty}, \] hence by \eqref{def-nu} and the dominated convergence theorem,
\[ \int_{0}^{1} \big\langle \Lambda_{\mathbf t(s)}, \mathtt f_{\ast} \big\rangle \, ds =  \lim_{\eps \rightarrow 0}\int_{0}^{1} \big\langle \Lambda_{\mathbf t(s)}^{(\eps)}, \mathtt f_{\ast} \big\rangle \, ds \geq 1. \]  
But our previous conclusion \eqref{will-contradict} gives that the integrand on the left hand side above is identically zero, hence the integral on the left is zero! This leads to an obvious contradiction, completing the proof of 
Theorem \ref{t-2}.
\end{proof}

\subsection{Continuity of $F^{[\mathtt f]}$} \label{cty-section}  
\noindent It remains to prove the continuity of the function $F^{[\mathtt f]}$ in \eqref{def-F}.
\begin{lemma} \label{cty-lemma}
For any $\mathtt f \in \mathcal V$, the function $F^{[\mathtt f]}: (0,1)^v \rightarrow \mathbb R$ defined in \eqref{def-F} is continuous. 
\end{lemma} 
\begin{proof} 
Let $\mathbf s \in (0,1)^v$. We need to show that $F^{[\mathtt f]}(\mathbf t) \rightarrow F^{[\mathtt f]}(\mathbf s)$ as $\mathbf t \rightarrow \mathbf s$. Since each entry of $\mathbf s$ is bounded away from zero, we may assume without loss of generality that $\mathbf t \in [\kappa, 1]^v$ for some $\kappa > 0$. The integral in \eqref{def-F} has been shown to be absolutely convergent in the proof of Proposition \ref{p-2}\eqref{nut-parta}, hence we may interchange the order of integration to write $F$ as 
\begin{align} F^{[\mathtt f]}(\mathbf t) &= \int_{\mathbb R} \widehat{\mu}(\xi) G(\xi, \mathbf t) d \xi \text{ where } \label{F-2} \\ G(\xi, \mathbf t) &:= \int_{\mathbb R^v} \Bigl[ \prod_{i=1}^{v} \widehat{\mu}(\eta_i) \Bigr] \widehat{\mathtt f}(-\eta_1 - t_1 \xi, \cdots, -\eta_v - t_v \xi) d\eta_1 \cdots d\eta_v. \label{def-G} \end{align}  
Since $\mathtt f \in \mathcal V$, there exists $R > 0$ such that $\text{supp}(\widehat{\mathtt f}) \subseteq [-R, R]^v$. Thus for every fixed $\xi$, the integrand in \eqref{def-G} is bounded, and the domain of integration in \eqref{def-G} is the compact set $\prod_{i=1}^{v}[-t_i \xi - R,-t_i \xi + R] \subseteq \mathbb R^v$. Hence by the dominated convergence theorem, the function $\mathbf t \mapsto G(\xi, \mathbf t)$ is continuous on $(0,1)^v$ for every fixed $\xi$.  Thus the integrand of the defining integral of $F^{[\mathtt f]}$ in \eqref{F-2} is a continuous function of $\mathbf t \in (0,1)^v$. We aim to show that this integrand is bounded above by an integrable function in $\xi$ that depends only on $\kappa$ but is otherwise independent of $\mathbf t \in [\kappa, 1]^v$. The desired continuity of $F^{[\mathtt f]}$ would then follow from the dominated convergence theorem. 
\vskip0.1in
\noindent The integrand in \eqref{F-2} has already been estimated in the proof of Proposition \ref{p-2} \eqref{nut-parta}. Specifically, the estimation leading up to \eqref{J-estimate}  shows that 
\begin{align*} \bigl| \widehat{\mu}(\xi) G(\xi, \mathbf t) \bigr| \leq |\widehat{\mu}(\xi)| \mathscr{J}(\xi; \mathbf t) &\leq C(R, \beta) ||\widehat{\mathtt f}||_{\infty} (1 + |\xi|)^{-\beta} \prod_{i=1}^{v} (1 + t_i |\xi|)^{-\beta} \\
&\leq C(R, \beta, \kappa) ||\widehat{\mathtt f}||_{\infty} (1 + |\xi|)^{-(v+1)\beta}.
\end{align*} 
The last expression is independent of $\mathbf t$ and integrable on $\mathbb R$, since $(v+1)\beta > 1$. This completes the proof. 
\end{proof} 
\section{Cantor-type sets avoiding linear patterns}

\subsection{A random Cantor construction} \label{Cantor-construction-section} 
We now turn to the complementary issue of constructing large sets avoiding linear patterns. An important tool here is a small variant of a Cantor-like construction due to Shmerkin \cite{S17}, which yields such avoiding sets. We begin by recalling the basic features of this construction from \cite[Section 2.1]{S17}.
\vskip0.1in
\noindent Let $\{L_n : n \geq 1\}$ and $\{M_n : n \geq 1\}$ be fixed sequences of positive integers, with $1\le L_n\le M_n$ and $2\le M_n$ for all $n\in\nn$. Using the notation $[M] := \{0, 1, \cdots, M-1\}$, we set 
\[
\Sigma_n := \Bigl\{\mathbf j:=(j_1,\ldots,j_n):\ j_k\in [M_k]  \text{ for each } k\in\{1,\ldots,n\} \Bigr\} \; \text{ and } 
\; \Sigma^*=\bigcup_{n=0}^\fz \Sigma_n.
\] 
We note that $\Sigma_0$ consists of the empty word $\emptyset$.
Each multi-index $\mathbf j=(j_1,\ldots,j_n)\in\Sigma_n$ is associated with an interval
\begin{equation} \label{basic-interval}  I_{\mathbf j}:=\left[\sum_{k=1}^n\frac{j_k}{M_1\cdots M_k},
\frac{1}{M_1\cdots M_n}+\sum_{k=1}^n\frac{j_k}{M_1\cdots M_k}\right] \subsetneq [0,1]. \end{equation} 
The Cantor set $E$ that we construct will be the countable intersection of a nested sequence of sets $E_n$, with \begin{equation} E_1 \supsetneq E_2 \supsetneq \cdots \supsetneq E_n \supsetneq E_{n+1} \supsetneq \cdots \supsetneq E, \qquad E = \bigcap_{n=1}^{\infty} E_n. \label{def-Cantor} \end{equation}  Each set $E_n$ will be a finite union of intervals of the form $I_{\mathbf j}$, for a choice of multi-indices $\mathbf j \in \Sigma_n$ to be specified.  
\vskip0.1in
\noindent Let $\{X_{\mathbf j}: \mathbf j \in \Sigma^{\ast}\}$ be a family of independent random sets obeying the following properties:
\begin{enumerate}[(I)]
\item For each $n \geq 0$ and $\mathbf j \in \Sigma_n$, the set $X_{\mathbf j}$ is a subset of $[M_{n+1}]$ with $\#(X_{\mathbf j})=L_{n+1}$, almost surely.
\item For each $a\in [M_{n+1}]$ and $\mathbf j \in \Sigma_n$, $\pp(a\in X_{\mathbf j})=L_{n+1}/M_{n+1}$.
\end{enumerate} \label{indep-criteria} 
As in \cite{S17}, we do not assume that $X_{\mathbf j}$ is chosen uniformly among all subsets of $[M_{n+1}]$ of size $L_{n+1}$. Each realization of the sequence of random sets $\{X_{\mathbf j} : \mathbf j \in \Sigma^{\ast}\}$ yields a sequence of multi-indices $\{\cj_n : n \geq 1\}$ in a natural way. Specifically, for every $n \geq 1$, 
$$\cj_n :=\{(j_1,\ldots,j_n) \in \Sigma_n:\ j_{k+1}\in X_{j_1\ldots j_k}
\mbox{ for all } k=0,\ldots,n-1\} \subseteq \Sigma_n.$$
Each collection $\cj_n$ in turn leads to a choice of basic intervals $\{I_{\mathbf j} : \mathbf j \in \cj_n\}$, where $I_{\mathbf j}$ is defined as in \eqref{basic-interval}. We use these to define the Cantor iterates $E_n$ and the Cantor-like set $E$:
\[ E_n = \bigcup_{\mathbf j \in \cj_n} I_{\mathbf j} \quad \text{ and } \quad E = \bigcap_{n=1}^{\infty} E_n. \] 
The recursive definition of $\cj_n$ ensures that the projection of $\cj_{n+1}$ onto the first $n$ coordinates yields $\cj_n$. In terms of the construction, this means that $E_{n+1} \subsetneq E_n$, as claimed in \eqref{def-Cantor}. From the assumptions (I) and (II), we have
$$|\cj_n|=L_1\cdots L_n \mbox{ almost surely,}$$
and that for every fixed multi-index $\mathbf j \in \Sigma_n$, 
\begin{equation}\label{pp}
\pp((\mathbf j, j_{n+1})\in\cj_{n+1})={\bf1}_{\cj_n}(\mathbf j) \frac{L_{n+1}}{M_{n+1}}
\mbox{ for all } j_{n+1} \in\{0,1,\ldots,M_{n+1}-1\}, 
\end{equation}
where ${\bf 1}_A(\cdot)$ the indicator function of the set $A$. Let us define functions
\begin{equation*}
\mu_n (x) :=\frac{M_1\cdots M_n}{L_1\cdots L_n}\sum_{\mathbf j \in\cj_n}{\bf1}_{I_\mathbf j}(x),
\end{equation*}
so that $\mu_n$ is a probability density function on $E_n$. By the Carath\'eodory Extension Theorem, there exists a probability measure $\mu$ supported on $E$ such that
\begin{equation} \label{mu} \mu(I_\mathbf j)=\mu_n(I_\mathbf j)=\frac1{L_1\cdots L_n}\mbox{ for all }\mathbf j\in\cj_n. \end{equation} 
In particular, $\mu_n\to\mu$ weakly.
\vskip0.1in
\noindent The randomness built into the construction of $\mu$ allows it to have maximal Fourier decay subject to the Hausdorff dimension of its support $E$. This is the content of the following theorem during to Shmerkin \cite[Theorem 2.1]{S17}.

\begin{theorem} \cite[Theorem 2.1]{S17} \label{t-A}
Let $\mu$ be the random Cantor measure supported on the Cantor-like set $E$, as described above. Suppose that the sequence $\{M_n : n \geq 1\}$ satisfies
\begin{equation}\label{log1}
\lim_{n\to\fz}\frac{\log M_{n+1}}{\log(M_1\cdots M_n)}=0.
\end{equation}
Fix any $\sz>0$ such that
\begin{equation}\label{log2}
\sz<\liminf_{n\to\fz}\frac{\log L_1\cdots L_n}{\log(M_1\cdots M_n)}.
\end{equation}
Then almost surely there exists a constant $C_{\sigma}>0$ such that
\begin{equation} |\wh\mu(k)|\le C_{\sigma} (1 + |k|)^{-\sigma/2}\mbox{ for all }k\in\zz. \label{log3} \end{equation} 
\end{theorem}


\subsection{Avoiding sets in the integers} 
We observe that the construction described in Section \ref{Cantor-construction-section} is very general, in the sense that under \eqref{log1}, any choice of sets $\{X_{\mathbf j}: \mathbf j \in \Sigma^{\ast} \}$ obeying (I) and (II) would lead almost surely to a probability measure $\mu$ that meets the Fourier decay condition \eqref{log3}. In our applications, the set $E$ will need to avoid certain rational linear patterns; hence the corresponding avoidance features have to be built into the sets $X_{\mathbf j}$. The building blocks for the sets $X_{\mathbf j}$ are certain deterministic constructions that ensure avoidance of nontrivial solutions of certain translation-invariant equations with integer coefficients.
In this section, we describe the construction of such sets in the integers, which will later be transferred to the continuum. These constructions are based on a fundamental idea of Behrend \cite{B46}. We state and prove here the version that we need, with special attention to the implicit constants that will become important later.  

\begin{proposition}\label{p-1}
Let $\cf$ and $\cf_v(\nn)$ denote the classes of translation-invariant linear functions with integer coefficients given by \eqref{def-cF} and \eqref{def-Fn} respectively. Fix any finite subcollection $\cg \st\cf$, and 
let $A$ be any large integer obeying
\begin{equation} A > 2 \max \lf\{m_0\;  \Biggl| \; \exists \nu \geq 2, f \in \cg \cap \cf_{\nu}(\mathbb N) \text{ with } f(x_0, \ldots, x_v):=m_0x_0-\sum_{i=1}^v m_ix_i \r\}. \label{def-A} \end{equation} 
Then, for every sufficiently large integer $N \geq N_0(A)$, there exists a set 
$Y_N = Y_N(\cg) \st [N]$ with
\begin{equation} 
\#(Y_N)>Ne^{-(A+4)\sqrt{\ln N}} \label{YN-large} 
\end{equation} 
and the following avoidance property. For every $f\in\cg \cap \cf_{\nu}(\mathbb N)$ and $x_0, \cdots, x_v \in Y_N$, 
\begin{equation} \label{no-zeros} f(x_0, x_1, \cdots, x_{\nu}) = 0  \quad \text{ implies } \quad x_0 = x_1 = \cdots = x_v. \end{equation} 
The constant $N_0(A)$ can be chosen to be $N_0(A) = 2^{C_0A^2}$ for some large absolute integer $C_0 > 0$. 
\end{proposition}
\noindent {\em{Remark:}} We observe that $Y_N(\cg)$ automatically precludes nontrivial zeros of $\cg$. Recall from the definition in page \pageref{def-nontrivial-zero} that a nontrivial zero vector of $f \in \cg \cap \cf_{v}(\mathbb N)$ has all distinct entries. In fact, Proposition \ref{p-1} yields the stronger conclusion that the only trivial zero of $f \in \cg$ admitted by $Y_N$ is the one with identical entries! 
\begin{proof}
Let $d, n\in\mathbb N$ be large integers to be specified shortly. Then any integer $y\in [(Ad)^n]$ admits a finite expansion with respect to the base $Ad$, i.e., 
there exists a unique sequence 
$\mathbf y = \mathbf y[y] =  (y_0,\ldots,y_{n-1})\in\{0,1,\ldots,Ad-1\}^n$ such that
$$y=y_0+y_1(Ad)+\cdots+y_{n-1}(Ad)^{n-1}.$$
For any integer $k \geq 0$, we set
\begin{equation} \label{def-Sk} 
\mathtt S_k(d,n):=\Bigg\{\ y\in [(Ad)^n]\; \Bigl| \;
\mathbf y[y] \in [d]^n \; \mbox{ and } \; \|\mathbf y[y]\|^2=\sum_{j=0}^{n-1}y_j^2=k
\Bigg\}.
\end{equation} 
The proof rests on the following two claims concerning these sets: 
\begin{align} 
&\text{ There exists $k_0 \in \{0, 1, 2, \ldots, n(d-1)^2\}$ such that $\# \bigl(\mathtt S_{k_0}(d, n) \bigr) > d^{n-2}/n$.  } \label{large-S} \\  
&\text{ Each $\mathtt S_k(d, n)$, if nonempty, avoids all nontrivial zeros of functions in $\cg$. } \label{avoid-S}
\end{align} 
Assuming these two claims for now, the proof is completed as follows. For sufficiently large $N\in \mathbb N$,
we choose $n:=\lfz\sqrt{\ln N}\rfz$ and $d\in \mathbb N$ such that
$(Ad)^n\le N<[A(d+1)]^n$, namely,
\begin{equation}\label{dan}
\frac{N^{1/n}}A-1<d\le\frac{N^{1/n}}A,
\end{equation}
and then set $Y_N:= \mathtt S_{k_0}(d,n)$. Then by \eqref{avoid-S}, the set $Y_N$ contains no nontrivial zero of any $f \in \cg$, establishing \eqref{no-zeros}.  
Further, using \eqref{large-S} we deduce that 
\begin{align}
\#Y_N&>\frac{d^{n-2}}n
>\frac{(N^{1/n}-A)^{n-2}}{nA^{n-2}}
=\frac{N^{1-2/n}}{nA^{n-2}}(1-AN^{-1/n})^{n-2} \notag \\
&=N \exp \Bigl[ \ln \left(N^{-\frac{2}{n}} (1-AN^{-\frac{1}{n}})^{n-2}/(nA^{n-2}) \right)  \Bigr] \notag \\  
&=N \exp \Biggl[-\frac{2\ln N}n - \ln n - (n-2) \Bigl[ \ln A-\ln(1-AN^{-1/n}) \Bigr] \Biggr] \notag  \\
&>Ne^{-(A+4)\sqrt{\ln N}}, \label{dan2}
\end{align}
where the last estimate holds for $N$ sufficiently large depending on $A$. This establishes \eqref{YN-large}, completing the proof of Proposition \ref{p-1}. We verify that $N = 2^{C_0A^2}$ allows for a choice of $d \in \mathbb N$ in \eqref{dan} and the estimate leading to \eqref{dan2}. 
\vskip0.1in
\noindent It therefore remains to prove the two claims \eqref{large-S} and \eqref{avoid-S}. We start with the former. By the definition \eqref{def-Sk} of $\mathtt S_k$, the entries of $\mathbf y \in \mathtt S_k$ are required to lie in $[d]$. Since $k = ||\mathbf y||^2$, this forces $0 \leq k \leq n(d-1)^2$. Thus the number of possible choices of $k$ such that $\mathtt S_k(d, n) \ne \emptyset$ is at most $n(d-1)^2+1$. Further, the sets $\{\mathtt S_k(d, n) : k \geq 1\}$ are disjoint, with 
\[ \bigsqcup_k \mathtt S_k(d, n) = \bigl\{y : \mathbf y[y] \in [d]^n \bigr\}, \quad \text{ hence } \quad \sum_k \#(\mathtt S_k) = d^n\] Therefore, by the pigeonhole principle,
there exists $k_0\in\{0,1,\ldots, n(d-1)^2\}$ such that
\begin{equation}\label{sk}
\#\bigl( \mathtt S_{k_0}(d,n) \bigr)\ge\frac{d^n}{n(d-1)^2+1}>\frac {d^{n-2}}n,
\end{equation}
proving \eqref{large-S}. 
\vskip0.1in
\noindent We now turn to proving \eqref{avoid-S}. 
Suppose that $(x_0, x_1, \cdots, x_{\nu}) \in \bigl[ \mathtt S_k(d, n) \bigr]^{\nu+1}$ is a zero of a function $f \in \cg$ of the form 
$f(x_0, \ldots, x_v):=m_0x_0-\sum_{i=1}^v m_ix_i$.  For $0 \leq i \leq \nu$, let $\mathbf x_i = (x_{i,0}, \ldots, x_{i, n-1})$ denote the vector of digits in $x_i$, in base $Ad$. Then $||\mathbf x_i|| = \sqrt{k}$ and  $0 \leq x_{i,j} < d$ for all $1 \leq i \leq \nu$ and $0 \leq j \leq n-1$. Using these digits, the equation
$m_0x_0=\sum_{i=1}^vm_ix_i$ can be rewritten as 
\[ m_0 \sum_{j=0}^{n-1} x_{0, j}(Ad)^{j} = \sum_{i=1}^{\nu} m_i \sum_{j=0}^{n-1} x_{i,j}(Ad)^{j}, \; \text{ or } \; \sum_{j=0}^{n-1} \bigl( m_0 x_{0, j} \bigr) (Ad)^{j} = \sum_{j=0}^{n-1} \Bigl[ \sum_{i=1}^{\nu} m_{i} x_{i,j} \Bigr] (Ad)^j. \] 
Since $A > 2m_0 =2( m_1 + \cdots+ m_v)$, it follows that, for any $j\in\{0,\ldots,n-1\}$, we have that $0 \leq m_0 x_{0, j} < m_0d < Ad$, and also $0 \leq \sum_{i=1}^{v} m_i x_{i, j} < d \sum_{i=1}^{v} m_i < Ad$. The uniqueness of the digit expansion then ensures that  the vectors $\{\mathbf x_i : 0 \leq i \leq \nu\}$, all of which lie on the sphere of radius $\sqrt{k},$ additionally satisfy the linear relation:
$$ m_0\mathbf x_0=\sum_{i=1}^v m_i\mathbf x_i.$$
By the length condition and the triangle inequality, we find that
$$m_0\sqrt k=\|m_0\mathbf x_0\|=\lf\|\sum_{i=1}^vm_i\mathbf x_i\r\|
\le\sum_{i=1}^vm_i\|\mathbf x_i\|=m_0\sqrt k.$$
Since the left and right ends match in the displayed chain of inequalities above, we deduce that the triangle inequality used in the intermediate step must in fact be an equality. This occurs only if $\mathbf x_1,\ldots,\mathbf  x_v$ are proportional.
Since we already know that 
$\|\mathbf  x_1\|=\cdots=\|\mathbf x_v\| = \sqrt{k}$, this forces $\mathbf x_1=\cdots=\mathbf x_v$,
and hence $x_0=x_1=\cdots=x_v$. This is a trivial solution of $f=0$.
Thus $\mathtt S_k(d,n)$ contains no nontrivial solution of $f=0$, for any $f\in\cg$. This is the desired conclusion \eqref{avoid-S}.  
\end{proof}

\subsection{Partial avoidance in the continuum} \label{partial-avoidance-section} 
Let $\cg \subsetneq \cf$ be a finite collection of translation-invariant linear equations with integer coefficients. In Proposition \ref{p-1}, we obtained a large subset of the integers that simultaneously avoids all nontrivial zeros of functions in $\cg$. Using these discrete building blocks and adapting an idea in  \cite[Lemma 2.2]{S17}, we now construct sets on the real line that embody a partial avoidance feature. Specifically, the sets that we construct in this section are unions of specially chosen small intervals. Any single interval contains nontrivial solutions of  $\cg$ of course, but the main conclusion is that points from distinct intervals fail to constitute nontrivial solutions of $\cg$. In the next section, we will repeat this argument on many scales to construct a Cantor-like set on $[0,1]$ with the full avoidance feature, as required by Theorem \ref{t-1}.  
\vskip0.1in   
\noindent Given $M\in\mathbb N$, $j\in [M]$, and $Y\st [M]$, we set 
\begin{equation} \label{def-I0}
\mathtt I_{M,j}:= \Bigl[\frac{j}{M},\frac{j+1}{M} \Bigr]\st \mathbb R/\mathbb Z \quad \text{ and } \quad \mathbb I(Y):=\bigcup_{j\in Y} \mathtt I_{M,j}\st [0,1]. \end{equation}  Thus $\mathbb I(Y)$ is the union of intervals in $[0,1]$, each of length $M^{-1}$, whose left endpoints are of the form $Y/M$. If  $j \in \mathbb Z$, then we define \[ \mathtt I_{M, j} := \mathtt I_{M, j'}, \quad \text{ where } j' \in [M] \text{ and }  j' = j \text{(mod $M$)}. \]  Thus, if $\ell \in \mathbb Z$, then $\mathbb I(Y+ \ell)$ is defined as 
\begin{equation}  \mathbb I(Y + \ell) := \bigcup_{j} \Bigl\{ \mathtt I_{M,j} : j \in Y+\ell (\text{mod } M) \Bigr\} \subseteq [0,1]. \label{def-I}  \end{equation}  

%
%

\begin{lemma}\label{l-1}
Given any finite set $\cg \st\cf$ as in Proposition \ref{p-1}, let the constant $A$ be as in \eqref{def-A}. 
Then,
for every sufficiently large integer $M \geq M_0(A)$, there exists a subset
$U_M = U_M(\cg) \subseteq [M]$ such that
\begin{enumerate}
\item[{\rm (i)}] $\#(U_M)>Me^{-(3A+4)\sqrt{\ln M}}$, 
\item[{\rm (ii)}] If a finite collection of points $\bigl\{x_0, x_1, \ldots, x_v \bigr\}\st \mathbb I(U_M+\ell)$ obeys the equation 
\begin{equation}  f(x_0, \cdots, x_v) =0 \quad \text{ for
some $f\in\cg \cap \cf_v(\mathbb N)$ and $\ell\in [M]$,}  \label{f-zero} 
\end{equation} 
then there exists $j\in U_M$ such that $\{x_0, x_1, \cdots, x_v \} \st \mathtt I_{M,j'}$, where $j' \in [M]$ and $j' = j+\ell$ (mod $M$).
\item[{\rm (iii)}] Elements of $U_M$ are at least $A$-separated modulo $M$, i.e. $|x - x' (\text{mod $M$})| \geq A$ for all $x, x' \in U_M$, $x \ne x'$. In particular, $U_M$ does not contain any two consecutive integers. 
\end{enumerate}
The constant $M_0(A)$ can be any integer larger than $(A+1)^2 2^{C_0A^2}$, where $C_0$ is the absolute constant in the statement of Proposition \ref{p-1}. For instance, the choice of $M_0(A) = 2^{C_0A^4}$ works. 
\end{lemma}

\begin{proof}
Let $\wz M:=\lfz\frac M{(A+1)^2}\rfz+1$.
If $M \geq M_0(A) > (A+1)^2 2^{C_0A^2}$, then by Proposition \ref{p-1}, there exists a set $Y_{\wz M}\st [\widetilde{M}]$ such that
\[ \# \bigl(Y_{\wz M} \bigr)>\wz M \exp \Bigl[-(A+4)\sqrt{\ln \wz M} \Bigr] \]
and $Y_{\wz M}$ contains no nontrivial zero of any $f\in\cg$.
Let $U_M:=A Y_{\wz M}:=\{Ax:\ x\in Y_{\wz M}\}$. 
The separation condition required by part (iii) of the lemma follows from this definition. 
\begin{align*} 
\#(U_M)=\#(Y_{\wz M}) &> \wz M \exp \Biggl[-(A+4)\sqrt{\ln \wz M} \Biggr] \geq \frac{M}{(A+1)^2} \exp \Bigl[-(A+4)\sqrt{\ln \wz M} \Bigr] \\ & > M \exp \Bigl[-2\ln (A+1)-(A+4)\sqrt{\ln M} \Bigr] \\ &>M \exp \Bigl[ -(3A+4)\sqrt{\ln M} \Bigr].
\end{align*} 
This concludes the proof of part (i). 
\vskip0.1in
\noindent Suppose now that $\{x_0, x_1, \ldots, x_v\} \subseteq \mathbb I(U_M + \ell)$ for some $\ell \in [M]$. Then for every $0 \leq i \leq v$, there exist 
$j_i\in Y_{\wz M}$ and $\dz_i\in[0,1)$ such that
\begin{equation*}
\left\{\displaystyle
\begin{array}{l@{\qquad}l}
\displaystyle x_i=\frac{Aj_i+\ell+\dz_i}M,& \mbox{ if }\ Aj_i+\ell<M,\\
\displaystyle x_i=\frac{Aj_i+\ell+\dz_i}M-1,& \mbox{ if }\ Aj_i+\ell\ge M,
\end{array}
\right. 
\end{equation*} 
This means that 
\begin{equation} 
\label{xi-exp} \qquad x_i = \frac{Aj_i+\ell+\dz_i}M - \chi_i , \qquad \text{ where } \qquad \chi_i= \begin{cases} 0 &\text{ if }  Aj_i+\ell< M,  \\ 1 &\text{ if }  Aj_i+\ell\ge M. \end{cases}
\end{equation}
Set $\ca:=\{i\in\{1,\ldots,v\}:\, \chi_i = 1\}.$
\vskip0.1in
\noindent If, in addition, $(x_0, x_1 \cdots, x_v)$ obeys the equation \eqref{f-zero}, with \[ f(x_0, \cdots, x_v) = m_0x_0 - \sum_{i=1}^vm_ix_i \]  for some $m_0, m_1, \cdots, m_v \in \mathbb N$ satisfying $m_0=\sum_{i=1}^vm_i$, then substituting \eqref{xi-exp} into this equation yields 
\begin{equation}\label{aa}
A \Biggl[ \sum_{i=1}^vm_ij_i - m_0j_0 \Biggr] + \Biggl[ \sum_{i=1}^vm_i\dz_i -m_0\dz_0 \Biggr]
-\lf(\sum_{i\in\ca} m_i  - m_0 \chi_0\r)M=0.
\end{equation}
We will show momentarily that \eqref{aa} implies the equality \begin{equation} m_0j_0 = \sum_{i=1}^vm_ij_i, \quad \text{i.e.} \quad f(m_0, m_1, \cdots, m_v) = 0. \label{YM-eqn} \end{equation} Assuming this for now, the proof is completed as follows.  Since $f \in \cg$ and $\{j_i: 0 \leq i \leq v \} \st Y_{\wz M}$, 
it follows from Proposition \ref{p-1} that for \eqref{YM-eqn} to hold, $(j_0, \ldots, j_v)$ has to be the identically constant solution, i.e., $j_0=\cdots=j_v=j$.
This in turn implies that $\{x_i: 0 \leq i \leq v\} \st \mathtt I_{M, j + \ell}$,
which is the conclusion of Lemma \ref{l-1}.  
\vskip0.1in
\noindent To deduce \eqref{YM-eqn} from \eqref{aa}, we proceed by contradiction, and consider two cases. Suppose first that  $\chi_0m_0 = \sum_{i \in \ca} m_i$. Since we have assumed that \eqref{YM-eqn} does not hold, it follows that 
\[A\le\lf|A\lf(\sum_{i=1}^vm_ij_i-m_0j_0\r)\r| \; \text{ whereas } \; \Bigl|\sum_{i=1}^vm_i\dz_i-m_0\dz_0 \Bigr|< \frac{A}{2} \text{ for any choice of $\delta_i \in [0,1)$}. \]
Thus \eqref{aa} cannot be satisfied. Next, assume that $\chi_0m_0 \ne \sum_{i \in \ca} m_i$. Then 
\begin{align*} M \bigl|\sum_{i\in\ca}m_i-\chi_i m_0 \bigr| &\ge M,  \text{ but} \\ 
\Biggl|A \Biggl[ \sum_{i=1}^vm_ij_i - m_0j_0 \Biggr] + \Biggl[ \sum_{i=1}^vm_i\dz_i -m_0\dz_0 \Biggr] \Biggr|
 &\le A \max \Bigl(m_0j_0, \sum_{i=1}^{v}m_i j_i \Bigr) + \frac{A}{2}\\ &\leq A^2\wz M + \frac{A}{2}< A^2 \left[ \frac {M}{(A+1)^2} + 1\right] + \frac{A}{2}<M,
\end{align*} 
This again implies that \eqref{aa} cannot hold. 
\end{proof}


\subsection{Large sets avoiding a finite number of rational linear patterns} 
\subsubsection{Proof of Theorem \ref{t-1}} 
\begin{proof}
Let $\cf$ be the collection of translation-invariant linear equations defined in \eqref{def-cF}, and let $\cg$ be a finite sub-collection of it. We aim to prove the existence of a probability measure $\mu$
such that $E = \supp(\mu)$ contains no nontrivial zero of any $f\in\cg$
and, for any $\eps>0$, there exists a constant $C_{\eps} > 0$ such that 
\begin{equation} \label{mu-Salem}  |\wh\mu(k)|\leq C_{\eps} (1 + |k|)^{-1/2+\eps}\mbox{ for all }k\in\zz. \end{equation} 
For this, we will follow the random Cantor construction in Section \ref{Cantor-construction-section}, with the random sets $\{X_{\mathbf j}: \mathbf j \in \Sigma^{\ast} \}$ chosen using Proposition \ref{p-1} and Lemma \ref{l-1}. 
\vskip0.1in
\noindent For sufficiently large integers $n\in \mathbb N$, let $U_n \subsetneq [n]$ be the set specified in Lemma \ref{l-1}. The conclusions of Lemma \ref{l-1} ensure that 
$\ln(\#U_n)/{\ln n}\to 1$ as $n\to\fz$. For any $n\in\nn$, we set 
\[  M_n :=N_0(A) + n  \quad \text{ and } \quad L_n := \#(U_{M_n}).  \] 
Such a choice of parameters obeys the condition \eqref{log1}. Further, 
$$\lim_{n\to\fz}\frac{\ln (L_1\cdots L_n)}{\ln(M_1\cdots M_n)}=1, \quad \text{ so that } \sigma = 1 - 2 \eps \text{ obeys \eqref{log2} for any $\eps > 0$.}$$
Let $\{\ell_\mathbf j:\ \mathbf j\in\Sigma^*\}$
be a sequence of independent random variables, with $\ell_\mathbf j$ distributed 
uniformly in $[M_{n+1}]$ for $\mathbf j\in\Sigma_n$. We set
\begin{equation*} X_\mathbf j:=U_{M_{n+1}}+\ell_\mathbf j \mod M_{n+1}, \ \mbox{ for any }\ \mathbf j\in\Sigma_n. \label{def-Xj} \end{equation*} 
As in \cite{S17}, one can verify that the random sets $\{X_\mathbf j: \mathbf j \in \Sigma^{\ast} \}$ are independent and satisfy the criteria (I) and (II) on page \pageref{indep-criteria}. 
\vskip0.1in
\noindent Let $\mu$ be the random Cantor measure defined in \eqref{mu}. By Theorem \ref{t-A}, $\mu$ obeys \eqref{mu-Salem}. 
It remains to show that $E = \supp \mu$ contains no nontrivial zero of any $f\in\cg$.
Indeed, suppose that $\{x_0, x_1, \cdots, x_v\}\subseteq E$ obeys $f(x_0, x_1, \cdots, x_v) = 0$ for some $f \in \cg$. Since the length $\delta_n = (M_1 M_2 \cdots M_n)^{-1}$ of the $n$-th level basic intervals in the Cantor construction shrinks to zero as $n \rightarrow \infty$, and since the points $x_i$ are distinct by definition, there exists a largest index $n \in \mathbb N$ such that $\{x_0, x_1, \cdots, x_v \} \subseteq I_\mathbf j \subseteq E_n$ for some $\mathbf j = \mathbf j_n \in \Sigma_n$. Suppose that $I_{\mathbf j} = \alpha_{\mathbf j} + [0,\delta_n]$.  
\vskip0.1in
\noindent The Cantor-like construction in Section \ref{Cantor-construction-section} dictates that at the $(n+1)^{\text{th}}$ step, $I_{\mathbf j}$ is decomposed into $M_{n+1}$ equal subintervals, of which only those corresponding to $U_{M_{n+1}} + \ell_{\mathbf j}$ are chosen. Thus $I_{\mathbf j} \cap E_{n+1}$ is a union of intervals of the form 
\[ \alpha_{\mathbf j} + \delta_n \Bigl[ \frac{j_{n+1}}{M_{n+1}}, \frac{j_{n+1}+1}{M_{n+1}}\Bigr], \quad \text{ with } j_{n+1} \in U_{M_{n+1}} + \ell_{\mathbf j} \text{ (mod  $M_{n+1}$)}.  \] 
This set is an affine copy of $\mathbb I(U_{M_{n+1}} + \ell_{\mathbf j})$ defined in \eqref{def-I0}. Since $\{x_0, x_1, \ldots, x_v\} \subseteq I_{\mathbf j} \cap E_{n+1}$, and the zero set of $f$ is closed under affine transformations, there exists an affine copy of $\{x_0, x_1, \ldots, x_v\}$, say $\{ x_0', x_1', \ldots, x_v'\}$, that is contained in $\mathbb I(U_{M_{n+1}} + \ell_{\mathbf j})$ and obeys \eqref{f-zero} with $\ell = \ell_{\mathbf j}$. It now follows from Lemma \ref{l-1}(ii)
that there exists $j_{n+1} \in U_{M_{n+1}}$ such that \[ \{x_0', x_1', \ldots, x_v' \} \subseteq \mathtt I_{M_{n+1}, j_{n+1}'} \quad \text{ where } \quad j_{n+1}' =  j_{n+1} + \ell_{\mathbf j} \text{ (mod $M_{n+1})$}. \] After an affine transformation, this is equivalent to saying that $\{x_0, x_1, \cdots, x_v \} \subset I_{\mathbf j_{n+1}}$ where $\mathbf j_{n+1} = (\mathbf j_n, j_{n+1}') \in \Sigma_{n+1}$.  But $I_{\mathbf j_{n+1}}$ is a basic interval of $E_{n+1}$. This contradicts the maximality of $n$, finishing the proof of Theorem \ref{t-1}.
\end{proof}


%
%

\subsection{Avoidance of nontrivial solutions with small diameter} 
\subsubsection{Proof of Theorem \ref{t-3}} 
\begin{proof}
We closely follow the proof strategy for Theorem \ref{t-1}. Our goal is to produce a probability measure $\mu$ that obeys \eqref{mu-Salem} for every $\eps > 0$ and whose support prevents nontrivial zeros of $\cf$ of arbitrarily small diameter.  
\vskip0.1in
\noindent For any $N\in\nn \setminus \{1, 2, \ldots, 5\}$, let us define
\[ \cq_N:=\Bigl\{f \in \cf: \exists v \geq 2 \text{ such that } f(x_0, \ldots, x_v):=m_0x_0-\sum_{i=1}^v m_ix_i, \; 
 m_0\le N/3 \Bigr\}. \]
Then $\cq_N$ is a finite sub-collection of $\cf$, for which the constant $A$ in \eqref{def-A} can be chosen to be $N$. Following Lemma \ref{l-1} with $\cg = \cq_N$, we set $M_N = 2^{C_0 N^4}$ and find a set $U_{M_N} = U_{M_N}(\cq_N) \subseteq [M_N]$ that obeys the conclusions of this lemma. 
In particular, part (iii) of this lemma ensures that $U_{M_N}$ does not contain any two consecutive integers modulo $M_N$. In addition, we have the following size condition from part (i):  
\[ L_{N} = \#(U_{M_N}) > M_N  \exp \bigl[-(3N + 4) \sqrt{\ln M_N} \bigr].\]
With this choice of $L_N$ and $M_N$, we ask the reader to verify that \eqref{log1} holds, and that $\sigma = 1 - 2 \eps$ verifies \eqref{log2} for every $\eps > 0$, since the limit inferior in \eqref{log2} is 1.  As in the proof of Theorem \ref{t-1}, we let $\{\ell_{\mathbf j} : \mathbf j \in \Sigma^{\ast} \}$ be a sequence of independent random variables, with $\ell_{\mathbf j}$ distributed uniformly in $[M_{N+1}]$ for $\mathbf j \in \Sigma_N$. We choose the random sets \begin{equation} \label{X-choices}  X_{\mathbf j} :=  U_{M_{N+1}} (\cq_{N+1}) + \ell_{\mathbf j}  \text{ mod } M_{N+1}  \text{ for } \mathbf j \in \Sigma_N. \end{equation} In other words, the discrete building blocks used at the $N^{\text{th}}$ step of the Cantor iteration are random translates of the set $U_{M_{N+1}}$, which avoids all zeros of $\mathcal Q_{N+1}$, except the constant vectors. It is easy to check that the sets $X_{\mathbf j}$ obey the criteria (I) and (II) on page \pageref{indep-criteria}. As a consequence, the random Cantor measure defined in \eqref{mu} obeys \eqref{mu-Salem}, by Theorem \ref{t-A}. Let $E$ denote the support of this measure. 
 \vskip0.1in
 \noindent It remains to prove that $E$ does not contain nontrivial solutions of $\cf$ with small diameter. Fix any $f \in \cf$. Then there exist an integer $v \geq 2$ and co-prime positive integers $m_0, m_1, \ldots, m_v$ such that $f(x_0, \cdots, x_v) = m_0 x_0 - (m_1 x_1 + m_2 x_2 + \cdots + m_v x_v)$. Since $\{\cq_N : N \geq 6\}$ is an increasing sub-collection that exhausts $\cf$, we can find $N \geq 6$ such that $f \in \cq_N \setminus \cq_{N-1}$. Set $\kappa = \kappa_N = \delta_N = (M_1 \cdots M_N)^{-1}$. Suppose now that $(x_0, \ldots, x_v) \in E^{v+1}$ obeys both the conditions in \eqref{small-diam}, with this choice of $\kappa$. We aim to show that 
\begin{equation} \label{diam-claim} \text{diam}(x_0, \ldots, x_v) \leq  \delta_{N+k} \text{ for every $k \geq 1$}. \end{equation}  
If we assume this for a moment, then letting $k \rightarrow \infty$ leads to diam$(x_0, x_1, \ldots, x_v) = 0$; in other words, $x_0 = x_1 = \cdots = x_v$, as required.
\vskip0.1in
\noindent Let us now turn to \eqref{diam-claim}, which we prove inductively in $k$. Since none of the sets $U_{M_n}$ contain consecutive integers mod $M_n$, the $n^{\text{th}}$ Cantor iterate $E_n$ does not contain adjacent basic intervals of length $\delta_n$, for any $n \geq 1$. Thus the assumption diam$(x_0, \ldots, x_v) < \kappa = \delta_N$ from \eqref{small-diam} means that there exists $\mathbf j \in \Sigma_N$ for which $\{x_0, \ldots, x_v\} \subseteq I_{\mathbf j} \subseteq E_N$. At step $N$ on the construction, the interval $I_{\mathbf j}$ is replaced by an affine copy of $\mathbb I(U_{M_{N+1}} + \ell)$ for some $\ell \in [M_{N+1}]$. Applying Lemma \ref{l-1}(ii) with $\cg = \cq_{N+1}$ and arguing as in the proof of Theorem \ref{t-1}, we deduce that there exists $ j_{N+1} \in [M_{N+1}]$ such that  $\{x_0, \ldots, x_v\} \subseteq I_{\mathbf j, j_{N+1}} \subseteq E_{N+1}$, and hence diam$(x_0, \ldots, x_v) \leq \delta_{N+1}$. 
\vskip0.1in
\noindent Let us proceed to the inductive step. At step $k$ of the induction (which corresponds to the $(N+k-1)^{\text{th}}$ stage of the Cantor iteration), we start with the inductive assumption that the points $x_0, \cdots, x_v$ all lie in a single basic interval $I_{\mathbf j_{N+k-1}} \subseteq E_{N+k-1}$ for some $\mathbf j_{N+k-1} \in \Sigma_{N+k-1}$. The choice of sets \eqref{X-choices} ensures that at step $(N+k-1)$ of the iteration, there exists an integer $\ell$ such that $\{x_0, \ldots, x_v \}$ is contained in an affine copy of $\mathbb I(U_{M_{N+k}} + \ell)$.  Applying Lemma \ref{l-1}(ii) with $\cg = \cq_{N+k}$, we conclude that the points $x_0, x_1, \ldots, x_v$ all lie in a single basic interval $I_{\mathbf j_{N+k}}$ of $E_{N+k}$, for some $\mathbf j_{N+k} \in \Sigma_{N+k}$. This completes the inductive step. 
 \end{proof}  

\section{Variations on a theme: more partial avoidance results}
In the remainder of this article, we turn our attention to Theorems \ref{t-4}, \ref{t-6} and \ref{t-5}, i.e., the construction of large sets that avoid nontrivial solutions of certain translation-invariant trivariate linear equations whose coefficients are not necessarily rational. The proofs of all these results follow the broad strokes of Theorem \ref{t-1}. Specifically, they rely on a two-part scheme that first constructs a partially avoiding set in the continuum which is then replicated at many scales to achieve full avoidance. The distinctions lie in the description of the building blocks with the partial avoidance feature. 

%
\subsection{Equations with coefficients lying in finitely many intervals} 
Given $M \in \mathbb N$, $Y \subseteq [M]$ and $\ell \in \mathbb Z$, let us recall from \eqref{def-I0} and \eqref{def-I} the definitions of $\mathtt I_{M,j}$ and $\mathbb I(Y + \ell) \subseteq \mathbb R/\mathbb Z$.   
\begin{lemma}\label{l-4}
Fix any $p\in\nn$ with $p\ge2$.
Then, for sufficiently large $M\in \mathbb N$ such that $\kz=\frac{p}{M} \leq \frac{1}{2p}$,
there exists
$V_M\st [M]$ such that
\begin{enumerate}
\item[{\rm (i)}] $\#V_M>M \exp \bigl[ -(5p+4)\sqrt{\ln M} \bigr]$, 

\item[{\rm (ii)}]
if $\{x,y,z\}\st \mathbb I(V_M+\ell)$ is a solution of 
$$tx+(1-t)y=z\quad\mbox{ for any }
\quad t\in\bigcup_{q=1}^{p-1}\lf(\frac qp-\kz,\frac qp+\kz\r)$$
 for some $\ell\in [M]$, 
then there exists $j\in V_M$ such that $\{x,y,z\}\st \mathtt I_{M,j+\ell}$.
\end{enumerate}
\end{lemma}

\begin{proof}
Given $p \in \mathbb N \setminus \{1\}$, let us define 
\[ \cg := \Bigl\{f: \mathbb R^3 \rightarrow \mathbb R \Bigl| f(x, y, z) = pz - qx - (p-q) y, \; 1 \leq q \leq p-1, q \in \nn \Bigr\} \subseteq \cf, \] 
so that the choice $A = 2p+1$ obeys \eqref{def-A}. Set $\wz M:=\lfz\frac M{16p^2}\rfz+1$. 
Applying Proposition \ref{p-1} with this choice of $\cg$, $A$ and sufficiently large $M$, we obtain a set 
\begin{equation}  Y_{\wz M}\st [\wz M] \quad \text{ such that } \quad \#Y_{\wz M}>\wz M \exp\Bigl[ -(2p+5)\sqrt{\ln \wz M} \Bigr], \label{Y-size}  \end{equation}  
with the property that whenever $\{x, y, z\} \subseteq Y_{\wz M}$ obeys the equation
\begin{equation} \label{Y-avoid} 
tx+(1-t)y=z \mbox{ for some } t\in \Bigl\{\frac 1p,\cdots, \frac{p-1}p\Bigr\}, \text{ one must have } x = y = z.  
\end{equation}
Set $V_M:=2p Y_{\wz M} =\{2px:\ x\in Y_{\wz M}\} \subseteq [M]$.
Then it follows from \eqref{Y-size} that 
\begin{align*} 
\#V_M=\# Y_{\wz M} &> \wz M \exp \bigl[ -(2p+5)\sqrt{\ln \wz M} \bigr] \\ 
& >M \exp \bigl[ -\ln (16p^2)-(2p+5)\sqrt{\ln M} \bigr] \\ &>M \exp \bigl[ -(5p+4)\sqrt{\ln M} \bigr], \quad {\text{provided $M$ is large enough.}} 
\end{align*} 
This verifies part (i) of the lemma. 
\vskip0.1in
\noindent For part (ii), suppose that $\{x,y,z\}$ is a triple of distinct points in $\mathbb I(V_M+\ell)$ that solves the equation 
\begin{equation} \label{eqn-e} 
\Bigl(\frac qp+\eps \Bigr)x+\Bigl(\frac {p-q}p-\eps\Bigr)y=z
\end{equation} 
for some $\ell\in [M]$, $q\in\{1,\ldots,p-1\}$ and $\eps \in(-\kz,\kz)$.  This means that there exist integers 
$\{j_x,j_y,j_z\}\st Y_{\widetilde{M}}$ and $\dz_x,\dz_y,\dz_z\in[0,1)$ such that
\begin{equation}
\left\{\displaystyle
\begin{array}{l@{\qquad}l}
\displaystyle x=\frac{2pj_x+\ell+\dz_x}M,& \mbox{ if }\ 2pj_x+\ell<M,\\
\displaystyle x=\frac{2pj_x+\ell+\dz_x}M-1,& \mbox{ if }\ 2pj_x+\ell\ge M.
\end{array}
\right. \label{rep-x} 
\end{equation}
We set 
\[ \chi_x= \begin{cases} 0  &\text{ if $2pj_x+\ell< M$},  \\ 1 &\text{ if $2pj_x+\ell\ge M$.} \end{cases} \]
The integers $j_y, j_z \in Y_{\widetilde{M}}$ and the binary counters $\chi_y, \chi_z \in \{0,1\}$ are defined similarly for $y$ and $z$. Substituting \eqref{rep-x} into \eqref{eqn-e}  and simplifying,   we obtain 
\begin{equation} \label{pqm}
\begin{aligned}
\Bigl(\frac qp+\eps \Bigr)(2pj_x+\dz_x)+ \Bigl(\frac {p-q}p-\eps \Bigr)&(2pj_y+\dz_y)-(2pj_z+\dz_z) \\ &=\Bigl[\Bigl(\frac qp+\eps \Bigr)\chi_x+ \Bigl(\frac {p-q}p-\eps\Bigr)\chi_y-\chi_z\Bigr]M.
\end{aligned}
\end{equation}
We claim that if \eqref{pqm} holds, then we must have $\chi_x = \chi_y = \chi_z$. Indeed, if this were not the case, then the right hand side of \eqref{pqm} would be nonzero, with
\begin{equation}  \Bigl |\Bigl(\frac qp+\eps \Bigr)\chi_x+\Bigl(\frac {p-q}p-\eps\Bigr)\chi_y-\chi_z\Bigr|M>\frac Mp - M \kappa. \label{rhs-lower} \end{equation} 
On the other hand, for any $j \in Y_{\widetilde{M}}$ and any $\delta \in [0,1)$, we have that $2pj + \delta \leq 2p (\widetilde{M}-1) + 1 < 2p \widetilde{M}$, hence the left hand side of \eqref{pqm} is bounded above by 
\begin{align} 
\Bigl|\Bigl(\frac qp+\eps \Bigr)(2pj_x+\dz_x) &+\Bigl(\frac {p-q}p-\eps \Bigr)(2pj_y+\dz_y)-(2pj_z+\dz_z)\Bigr| \notag \\ 
&< \Bigl(\frac qp+\eps \Bigr) 2p \widetilde{M}  +  \Bigl(\frac {p-q}p-\eps \Bigr) 2p \widetilde{M} + 2p \widetilde{M}   \leq 4p \wz M <\frac M{2p}. \label{lhs-upper} 
\end{align} 
The estimates \eqref{rhs-lower} and \eqref{lhs-upper} substituted into \eqref{pqm} provide the desired contradiction, since $M$ has been chosen large enough to ensure $\kappa \leq \frac{1}{2p}$. In view of $\chi_x=\chi_y=\chi_z$, \eqref{pqm} becomes
\begin{align}
(2pj_z &+ \delta_z) = \Bigl(\frac{q}{p} + \eps \Bigr) (2pj_x + \delta_x) +  \Bigl(\frac{p-q}{p} - \eps \Bigr) (2pj_y + \delta_y); \; \text{rearranging this, we get} \notag \\   
\label{pqm1}
&2qj_x+2(p-q)j_y-2pj_z=\eps (2pj_y+\dz_y-2pj_x-\dz_x)+\dz_z- \frac qp \dz_x- \frac {p-q}p \dz_y.
\end{align}
The left hand side of \eqref{pqm1} is either zero or an even integer, whereas the right hand can be bounded as follows, 
\begin{align*} 
\Bigl|\eps(2pj_y+\dz_y-2pj_x-\dz_x) &+\dz_z- \frac qp \dz_x- \frac {p-q}p \dz_y\Bigr| \\  &\leq \Bigl|\eps(2pj_y+\dz_y-2pj_x-\dz_x)+\dz_z \Bigr| + \Bigl| \dz_z- \frac qp \dz_x- \frac {p-q}p \dz_y\Bigr| \\ & \leq 2\kappa p \widetilde{M} + 1 < 2, \text{ given our choice of $\kappa = \frac{p}{M}$ and $\widetilde{M}$.}  
\end{align*} 
Thus both sides of \eqref{pqm1} must vanish, i.e., $qj_x+(p-q)j_y-pj_z=0$.
Since $\{j_x,j_y,j_z\}\st Y_{\wz M}$ and obey $q x+ (p-q) y= pz$, it follows from \eqref{Y-avoid} that 
$j_x=j_y=j_z=j_0$.
This implies that  $j = 2p j_0 \in V_M$ and hence $\{x, y, z\}\st \mathtt I_{M, j + \ell}$,
which finishes the proof of Lemma \ref{l-4}.
\end{proof}

\subsection{Proof of Theorem \ref{t-4}} 
\begin{proof}
Let us fix an integer $p \geq 2$ and $\alpha \in (0,1)$. Using the prescription of Section \ref{Cantor-construction-section}, we aim to find a probability measure $\mu$
such that $E = \supp(\mu)$ contains no nontrivial solution of 
\begin{align} tx+(1-t)y=z \mbox{ for any } t\in\bigcup_{q=1}^{p-1}\Bigl(\frac qp-\kz,\frac qp+\kz\Bigr), \;\text{ and } \label{eqn-avoid}  \\
|\wh\mu(k)|\leq C_{\eps} (1 + |k|)^{-\az/2 + \eps}\mbox{ for all }k\in \zz \text{ and every } \eps > 0. \label{mu-F-decay} \end{align} 
We first choose $M\in\nn$ large enough so that \[ M \exp \bigl[ -(5p+4)\sqrt{\ln M} \bigr] >M^\az, 
\quad \text{ namely, }  \quad M>\exp \Bigl(\Bigl[\frac{5p+4}{1-\az}\Bigr]^2\Bigr). \]
Let $V_M\st [M]$ be as in Lemma \ref{l-4}
and $\kz(\az,p):=\frac{p}{M}$ as specified by the lemma. We then deduce from this lemma that
$\frac{\ln(\#V_M)}{\ln M}>\az$. For any $n\in\nn$, we set \[ M_n = M, \; L_n:=\#V_M, \; \text{ so that } \; \lim_{n\to\fz}\frac{\ln M_{n+1}}{\ln(M_1\cdots M_n)}=0
\text{ and } 
\lim_{n\to\fz}\frac{\ln (L_1\cdots L_n)}{\ln(M_1\cdots M_n)} \geq \az. \] 
Let $\{\ell_{\mathbf j}:\ \mathbf j \in\Sigma^*\}$ be a sequence of independent random variables, with $\ell_{\mathbf j}$ distributed 
uniformly in $[M_{n+1}]$ for $\mathbf j\in\Sigma_n$.
Set
$$X_{\mathbf j} := V_M +\ell_{\mathbf j} \mod M, \ \mbox{ for any } \mathbf j\in\Sigma_n.$$
Then the sets $X_{\mathbf j}$ are independent and satisfy the criteria (I) and (II) in page \pageref{indep-criteria}.
By Theorem \ref{t-1}, the measure $\mu$ defined in Section \ref{Cantor-construction-section}  satisfies \eqref{mu-F-decay}. 
\vskip0.1in
\noindent It remains to show that $E = \supp(\mu)$ contains no nontrivial solution of any of the equations in \eqref{eqn-avoid}. This is a repetition of similar arguments from Theorems \ref{t-1} and \ref{t-3}.   
In fact, if the triple $\{x,y,z\}\subseteq E$  constitutes a nontrivial solution of
some equation in \eqref{eqn-avoid} and $I_{\mathbf j}$ is the smallest interval in the construction containing all the three points $x,y,z$, it would follow from Lemma \ref{l-4}
that $\{x,y,z\}\in I_{\mathbf j, j_{n+1}}$ for some $(\mathbf j, j_{n+1}) \in \Sigma_{n+1}$, contradicting minimality. This finishes the proof of Theorem \ref{t-4}.
\end{proof}
\section{An uncountable collection of forbidden coefficients} 
%
\subsection{A general construction} \label{coeff-construction-section} 
In preparation for Theorem \ref{t-5}, we specify in this a section a general prescription for creating an uncountable set $\mathfrak D$. In the next two sections, we will use this construction to describe the set of forbidden coefficients $\mathfrak C$ and prove Theorem \ref{t-5}. 
\vskip0.1in  
\noindent Given integers $1 \leq P \leq R+2$, let $\mathscr{Q}_P$ denote the set of dyadic rationals of the form $\mathbb Z 2^{-P}$, and 
\begin{equation} \label{APR} \mathfrak A(P, R) := \bigcup_{q \in \mathbb Z} \Bigl[\frac{q}{2^P} - 2^{-R}, \frac{q}{2^P} + 2^{-R} \Bigr] = \Bigl\{x : \text{dist}(x, \mathscr{Q}_P)  \leq 2^{-R} \Bigr\}. \end{equation}   
Given an infinite sequence of positive integers $1 = P_1 < R_1 < P_2 < R_2 < \cdots < P_n < R_n < \cdots$, where all successive differences are at least 2, we set
\begin{equation} \label{D-def} 
\begin{aligned} \mathfrak D_1 &:= \Bigl[\frac{1}{2} - 2^{-R_1},  \frac{1}{2} + 2^{-R_1} \Bigr], \quad \mathfrak D_2 := \mathfrak D_1 \cap \mathfrak A(P_2, R_2), \\ \mathfrak D_n &:= \mathfrak D_{n-1} \cap \mathfrak A(P_n, R_n), \quad \text{ and } \quad \mathfrak D := \bigcap_{n=1}^{\infty} \mathfrak D_n.  \end{aligned}  
\end{equation} 
We observe that each $\mathfrak D_n$ is a finite union of closed intervals of length $2^{-R_n +1}$. As a result, $\mathfrak D$ is closed. 
\begin{lemma} \label{D-infinite-lemma} 
The set $\mathfrak D$ contains infinitely many rationals and uncountably many irrationals. 
\end{lemma} 
\begin{proof} 
Let $\pmb{\eps} = (\eps_1, \eps_2, \ldots)$ be an infinite binary sequence, with entries in $\{0,1\}$. We claim that every point of the form 
\begin{equation}  t(\pmb{\eps}) := \frac{1}{2} + \sum_{n=1}^{\infty} \frac{\eps_n}{2^{R_n+2}} \text{ with }  \pmb{\eps} \in \prod_{n=1}^{\infty} \{0,1\} \label{t-induction} 
\text{ lies in $\mathfrak D$. } 
\end{equation} 
The conclusion of the lemma is an easy consequence of this claim. The collection of binary sequences is uncountable, each sequence yields a unique $t(\pmb{\eps})$ due to the separation condition $R_{n+1} - R_n \geq 4$, and the sequences $\pmb{\eps}$ that are eventually zero yield infinitely many distinct (dyadic) rationals. 
\vskip0.1in
\noindent It remains to prove \eqref{t-induction}, which is equivalent to showing that $t = t(\pmb{\eps}) \in \mathfrak D_k$ for all $k \geq 1$. 
This is a consequence of the following statements, which we will prove momentarily using induction: for every $k \geq 1$,  
\begin{align} \label{where-is-t}  &t_k = t_k(\pmb{\eps}) := \frac{1}{2} + \sum_{n=1}^{k} \frac{\eps_n}{2^{R_n+2}} \in \mathfrak D_k \cap \mathscr{Q}_{P_{k+1}} \text{ and } \\  &t_{\ell} \in \mathfrak D_{k+1} \text{ for all }  \ell \geq k+1, \text{ hence $t \in \mathfrak D_{k+1}$.} \label{where-is-t-2}\end{align}  
Let us start with the base case $k = 1$, with  $t_1 = \frac{1}{2} + \frac{\eps_1}{2^{R_1+2}}$. For all $\ell \geq 1$, we have that  
\begin{equation} \label{tl-in-D1}
 \bigl| t_{\ell} - \frac{1}{2} \bigr| \leq \sum_{j=R_1+2}^{\infty} 2^{-j} \leq 2^{-R_1-1}; \; \text{ hence } t_{\ell} \in \mathfrak D_1 \text{ for all $\ell \geq 1$}. \end{equation}  
Further, $t_1$ can be written as 
\begin{align*} 
 &t_1 = \frac{2^{P_2-R_1-2}(\eps_1 + 2^{R_1+1})}{2^{P_2}}, \quad \text{ which shows that } t_1 \in \mathscr{Q}_{P_2}, \text{ and } \\ &\bigl| t_{\ell} - t_1\bigr| \leq \sum_{j=2}^{\infty} 2^{-R_2-j} \leq 2^{-R_2-1} < 2^{-R_2} \text{ for } \ell \geq 2, \end{align*} 
 hence $t_{\ell} \in \mathfrak A(P_2, R_2)$ for $\ell \geq 2$, according to the definition \eqref{APR}. Combining the last observation with \eqref{tl-in-D1}, we see that  each such $t_{\ell}$ lies in $\mathfrak D_1 \cap \mathfrak A(P_2, R_2) = \mathfrak D_2$. Since $t_{\ell} \rightarrow t$ and $\mathfrak D_2$ is closed, we conclude that $t \in \mathfrak D_2$, which completes the base case of the induction.
 \vskip0.1in
 \noindent For the inductive step, we assume that $t_{k-1} \in \mathfrak D_{k-1} \cap \mathscr{Q}_{P_{k}}$ and that $t, t_{\ell} \in \mathfrak D_k$ for all $\ell \geq k$. Since the class $\mathscr{Q}_P$ of dyadic rationals is closed under addition with $ \mathscr{Q}_{P_{k}} \subseteq  \mathscr{Q}_{P_{k+1}}$ and \[ t_k = t_{k-1} + \eps_k 2^{-R_k-2} = t_{k-1} + \frac{\eps_k 2^{P_{k+1} - R_k-2}}{2^{P_{k+1}}}, \] it is easy to see that $t_{k} \in \mathscr{Q}_{P_{k+1}}$. Since $t_k \in \mathfrak D_k$ by the induction hypothesis, \eqref{where-is-t} follows. 
 For $\ell \geq k+1$, we note that $|t_{\ell} - t_k| \leq \sum_{j=2}^{\infty} 2^{-R_{k+1}-j} \leq 2^{-R_{k+1}-1} < 2^{-R_{k+1}}$.  The earlier observation that $t_k \in \mathscr{Q}_{P_{k+1}}$ and the induction hypothesis $t_{\ell} \in \mathfrak D_k$ then imply that $t_{\ell} \in \mathfrak A(P_{k+1}, R_{k+1}) \cap \mathfrak D_k = \mathfrak D_{k+1}$ for all $\ell \geq k+1$. Letting $\ell \rightarrow \infty$ and using the fact that $\mathfrak D_{k+1}$ is closed, we obtain $t \in \mathfrak D_{k+1}$. This establishes \eqref{where-is-t-2} and completes the induction.    
\end{proof} 
\subsection{Construction of $\mathfrak C$} 
For $p \in \mathbb N$ and $\kappa > 0$, we will use the notation
\begin{equation} \mathscr{A}(p, \kappa) := \bigcup_{q=1}^{p-1} \Bigl( \frac{q}{p} - \kappa, \frac{q}{p} + \kappa \Bigr)  \label{union-pk} \end{equation}  
to denote the disjoint union of intervals appearing in Lemma \ref{l-4}. This lemma will be employed repeatedly to construct a set $\mathfrak C$ as in Section \ref{coeff-construction-section}. 
\vskip0.1in
\noindent For $n\in\nn$, let 
\begin{equation} \label{what -is-alphan}
\az_n:=\frac n{n+1}, \text{ so }  \alpha_n \nearrow 1. 
\end{equation} 
\begin{proposition} \label{C-construction-prop} 
There exist an infinite set \begin{equation}  \mathfrak C = \bigcap_{n=1}^{\infty} \mathfrak C_n  \quad \text{ of the form \eqref{D-def},} \label{C-def} \end{equation} 
increasing sequences of positive integers $\{R_n: n\geq 1\}$, $\{ N_n: n \geq 1\}$ and sets $V_{N_n} \subseteq [N_n]$ such that 
\begin{align}   \# V_{N_n} &>N_n\exp \bigl[-(5 \cdot 2^{R_{n-1}+2}+4)\sqrt{\ln N_n} \bigr]>N_n^{\az_n}, \label{V-large} \\ \mathfrak C_n &= \mathfrak C_{n-1} \cap \mathfrak A(R_{n-1}+2, R_n). \label{def-Cn} 
\end{align} 
 In addition, for every $\ell \in [N_n]$, the set $\mathbb I(V_{N_n} + \ell)$ partially avoids the coefficient set $\mathfrak C_n$ in the following sense. If a triple $\{x, y, z\} \subset \mathbb I(V_{N_n} + \ell)$ solves the equation \begin{equation}  tx + (1-t)y = z  \quad \text{ for any } t \in \mathfrak C_n  \label{xyz-solve} \end{equation}  then there exists $j \in V_{N_n}$ such that 
 \begin{equation} \label{part-avoid}  \{x, y, z\} \subset \mathtt I_{N_n, j+\ell}, \text{ a basic interval of $\mathbb I(V_{N_n} + \ell)$.} \end{equation}    
In particular, $\mathbb I(V_{N_n} + \ell)$ avoids the coefficient set $\mathfrak C$ defined in \eqref{C-def}, for every $n \geq 1$. 
\end{proposition} 
\begin{proof} We will use induction on $n$ to create $R_n$, $N_n$, $V_{N_n}$ and $\mathfrak C_n$ that obey \eqref{V-large}, \eqref{def-Cn}, and the partial avoidance criterion \eqref{xyz-solve} $\implies$ \eqref{part-avoid}. The set $\mathfrak C$ is then obtained using the defining formula \eqref{C-def}. This is consistent with the prescription in Section \ref{coeff-construction-section}, with $P_1=1$ and $P_n = R_{n-1}+2$ for $n \geq 2$. 
\vskip0.1in
\noindent Let us initialize $R_0= -1$ and $\mathfrak C_0 = [0,1]$. Applying Lemma \ref{l-4} with $p_1 = 2^{P_1} = 2^{R_0+2}=2$, we can find a large integer $N_1$,
a small positive constant $\kappa_1$ and a large set $V_{N_1}\subseteq [N_1]$ that partially avoids $\mathscr{A}(p_1, \kappa_1)$. More precisely, we can choose $N_1$ so that \eqref{V-large} holds, 
and if $\{x,y,z\}\st \mathbb I(V_{N_1}+\ell)$ is a solution of 
$$tx+(1-t)y=z\quad\mbox{ for some } t\in \mathscr{A}(p_1, \kappa_1) = \Bigl(\frac{1}{2} - \kappa_1, \frac{1}{2} + \kappa_1 \Bigr)$$
and for some $\ell\in [N_1]$, 
then Lemma \ref{l-4} (ii) dictates the existence of $j\in V_{N_1}$ such that $\{x,y,z\}\st \mathtt I_{N_1,j+\ell} \subseteq \mathbb I(V_{N_1}+\ell)$.  Set $R_1$ to be an integer such that $\varrho_1 := 2^{-R_1} < \kappa_1$ and define 
\[ \mathfrak C_1 := \Bigl [\frac 12-\varrho_1,\frac 12+\varrho_1\Bigr]. \]
This verifies \eqref{def-Cn} for $n=1$. Since $\mathfrak C_1 \subseteq \mathscr{A}(p_1, \kappa_1)$, the implication \eqref{xyz-solve} $\implies$ \eqref{part-avoid} follows from the avoiding property of $V_{N_1}$, concluding the base case of the induction. 
\vskip0.1in 
\noindent We arrive at the inductive step. Let us assume now that the conditions in the proposition have been verified for all $n \leq k-1$. At the $k^{\text{th}}$ step, we set $p_k = 2^{R_{k-1}+2}$. Invoking Lemma \ref{l-4} yields a large integer $N_k$, a small constant $\kappa_k > 0$ and a large set $V_{N_k} \subseteq [N_k]$ obeying the conclusions of the lemma; in particular, we can choose $N_k$ large enough to ensure \eqref{V-large} with $n=k$. 
Further, if $\{x,y,z\}\st \mathbb I(V_{N_k}+\ell)$ is a solution of 
$$tx+(1-t)y=z\quad \text{ for some } t\in  \mathscr{A}(p_k, \kappa_k) \text{ and $\ell\in [N_k] $} $$
 then there exists $j\in V_{N_k}$ such that $\{x,y,z\}\st \mathtt I_{N_k,j+\ell}$. Choosing $R_k \in \nn$ to satisfy $2^{-R_k} < \kappa_k$, and setting $\mathfrak C_k = \mathfrak C_{k-1} \cap \mathfrak{A}(R_{k-1}+2, R_k) \subseteq  \mathscr{A}(p_k, \kappa_k)$ verifies all the conditions of the proposition for $n=k$, completing the induction.  
\end{proof} 

\subsection{Proof of Theorem \ref{t-5}} 
\begin{proof} 
Let $\mathfrak C$ denote the set specified in Proposition \ref{C-construction-prop}. Since it is of the form \eqref{D-def} constructed in Section \ref{coeff-construction-section}, it follows from Lemma \ref{D-infinite-lemma} that $\mathfrak C$ contains infinitely many rationals and uncountably many irrationals. We now aim to show that there exists a probability measure $\mu$ on $[0,1]$ such that $E = \supp \mu$ contains no nontrivial solution of $t x+(1-t)y=z$ for any $t \in \mathfrak C$, and that for any $\eps>0$, there is a constant $C_{\eps} > 0$ such that 
\begin{equation} |\wh\mu(k)|\leq C_{\eps} (1 + |k|) ^{-1/2+\eps}\mbox{ for all }k\in\zz. \label{F-decay-again} \end{equation} 
\vskip0.1in
\noindent To construct $\mu$, we follow the same general pattern as the random construction in Section \ref{Cantor-construction-section}, the only distinction being that the parameters $M_n$ in that construction have to be chosen depending on the parameters $N_n$ in Proposition \ref{C-construction-prop} so as to obey \eqref{log1}. Given a fixed large absolute constant $C_0 \geq 100$, we first set $N_n = \exp[C_0^2 (n+1)^2 2^{2R_{n-1}}]$, so that the second inequality in \eqref{V-large} is satisfied. Given the fast-growing nature of $R_n$ and hence $N_n$, the obvious choice of $M_n = N_n$ does not meet the criterion \eqref{log1}, so we need to control the parameters $M_n$ more  carefully. To this end, let 
\[ b_k :=\lfz\frac{(k+1)\ln N_{k+2}}{\ln N_k}\rfz+1, \quad B_k := \sum_{i=1}^{k} b_i, \quad \text{ and } \quad L_n = \#(V_{M_n}), \text{ where } \]
\begin{equation} \label{def-Mn}
M_n= \begin{cases} 
N_1  &\text{ if } n\le B_1 = b_1,\\
N_m &\text{ if }\ m \in\nn \text{ and } B_{m-1}<n \le B_m.
\end{cases}
 \end{equation}
The independent random sets $\{X_{\mathbf j} : \mathbf j \in \Sigma^{\ast} \}$ are chosen as before. Namely, given a sequence $\{\ell_{\mathbf j}:\ \mathbf j\in\Sigma^*\}$
of independent random variables, with $\ell_{\mathbf j}$ distributed uniformly in $[M_{n+1}]$ for $\mathbf j \in\Sigma_n$, we set
$$X_{\mathbf j}:=V_{M_{n+1}}+\ell_{\mathbf j} \mod M_{n+1}, \ \mbox{ for any }\ \mathbf j\in\Sigma_n.$$
Then the sets $X_{\mathbf j}$ obey the independence criteria (I) and (II) on page \eqref{indep-criteria}. Let $\mu$ be as in \eqref{mu}. We will use Theorem \ref{t-A} to establish that $\mu$ has the Fourier decay property \eqref{F-decay-again}.  
\vskip0.1in
\noindent Let us check the hypothesis of Theorem \ref{t-A}. For any $B_{m-1} < n \leq B_m$, we have from the definition of $b_m$ that 
$$\frac{\ln M_{n+1}}{\ln(M_1\cdots M_n)}\leq \ln M_{n+1} \Biggl[ \sum_{j= B_{m-2}+1}^{B_{m-1}} \ln M_j \Biggr]^{-1}
\leq \frac{\ln N_{m+1}}{b_{m-1}\ln N_{m-1}}<\frac1{m},
$$
which goes to zero as $n \rightarrow \infty$. This verifies \eqref{log1}. For $B_{m-1} < n \leq B_m$, we also have from \eqref{V-large} and \eqref{def-Mn} that 
$$\frac{\ln L_n}{\ln M_n}> \frac{\ln \#(V_{N_m})}{\ln N_m} \geq \az_m=\frac m{m+1} \rightarrow 1 \text{ as } n \nearrow \infty.$$
This shows that the liminf in \eqref{log2} is 1, and hence we may choose $\sigma = 1 -2 \eps$ for any  $\eps > 0$. Substituting this into \eqref{log3} leads to \eqref{F-decay-again}.  
\vskip0.1in
\noindent It remains to prove the avoidance property of $E = \text{supp}(\mu)$. The argument here is the same as in previous Theorems \ref{t-1} and \ref{t-4}.  
Suppose that $\{x,y,z\}\st E$  is a nontrivial solution of $t x+(1-t)y=z$ for some $t \in \mathfrak C$. Let $I_{\mathbf j}$ be the smallest basic interval in the construction of $E$ that contains $\{x,y,z\}$. Consider the affine transformation $T$ that maps $I_{\mathbf j}$ onto $[0,1]$. If $\mathbf j \in \Sigma_n$, then the corresponding affine copy $\{ T(x), T(y), T(z)\}$ of $\{x, y, z\}$ is contained in $\mathbb I(V_{M_{n+1}} + \ell)$. Since $M_{n+1} = N_m$ for some $m$ and $\mathfrak C \subseteq \mathfrak C_m$ for all $m$, it follows from Proposition \ref{C-construction-prop} 
that $\{T(x), T(y), T(z) \}$ is contained in $\mathtt I_{N_m, j+\ell}$, one of the basic intervals of $\mathbb I(V_{M_{n+1}} + \ell)$. Applying the transformation $T^{-1}$ that brings $[0,1]$ back to $I_{\mathbf j}$, we find that $\{x,y,z\}\in I_{\mathbf j, j_{n+1}}$ for some $j_{n+1} \in [M_{n+1}]$, contradicting minimality of $I_{\mathbf j}$. This completes the proof of Theorem \ref{t-5}.
\end{proof}

\section{The coefficient set of badly approximable numbers} 
The goal of this section is to prove Theorem \ref{t-6}. While the basic strategy remains the same as our previous non-existence results (Theorems \ref{t-1} and \ref{t-4}), one of the key discretization steps, namely Behrend's principle (Proposition \ref{p-1}), no longer applies, due to the diophantine nature of $\mathcal  E_{c, \tau}$. We now need to choose the building blocks differently.
\subsection{A uniform set of intervals partially avoiding coefficients in $\mathcal E_{c, \tau}$}   
We will use the definitions of $\mathtt I_{M,j}$ and $\mathbb I(Y)$ from Section \ref{partial-avoidance-section}.  
\begin{lemma}\label{l-5}
Let us fix $0 < \tau, c \leq 1$ and  $\eps_0\in(0,\frac12)$. Then for all sufficiently large $M\in\nn$,
there exists a set $W_M\subseteq [M]$ with the following two properties: 
\begin{enumerate}
\item[{\rm (i)}] $\#(W_M) >(\frac{c \eps_0M}{20})^{\frac1{1+\tau}}$,

\item[{\rm (ii)}]
If $\{x,y,z\}\st \mathbb I(W_M+\ell)$ is a solution of 
$$tx+(1-t)y=z\quad\mbox{ for any }
\quad t\in\cec\cap(\eps_0,1-\eps_0)$$
 for some $\ell\in [M]$, 
then there exists $j\in W_M$ such that $\{x,y,z\}\st \mathtt I_{M,j+\ell}$.
\end{enumerate}
\end{lemma}

\begin{proof}
For large enough integers $M$, let us set 
\[ N:=\lfz(\frac{c \eps_0M}{10})^{\frac1{1+\tau}}\rfz, \quad R :=\lfz\frac {N^\tau}c\rfz+1 \quad \text{ and  } \quad W_M:=\{0,R,2R,\ldots, (N-1)R\} \subseteq [M]. \] 
Then
$$\#W_M=N>\lf(\frac{c \eps_0M}{20}\r)^{\frac1{1+\tau}},$$
verifying part (i) of the lemma. 
\vskip0.1in
\noindent Suppose now that the triple $\{x,y,z\} \subseteq \mathbb I(W_M + \ell)$ is a solution of 
some $tx+(1-t)y=z$ for some
$t\in\cec\cap(\eps_0,1-\eps_0)$. Then there exists some 
$j_x, j_y$ and $j_z \in \{0,\ldots, N-1\}$ and $\dz_x, \dz_y$ and $\dz_z\in[0,1)$ such that
\begin{equation}
\left\{\displaystyle
\begin{array}{l@{\qquad}l}
\displaystyle x=\frac{Rj_x+\ell+\dz_x}M,& \mbox{ if }\ R j_x+\ell<M,\\
\displaystyle x=\frac{Rj_x+\ell+\dz_x}M-1,& \mbox{ if }\ R j_x+\ell\ge M.
\end{array}
\right. \label{jx-def} 
\end{equation}
As before, we set 
\begin{equation} \label{chix-def} 
\chi_x= \begin{cases} 0  &\text{ if } R j_x+\ell< M, \\ 1 &\text{ if }  R j_x+\ell\ge M, \end{cases}. \end{equation}  The binary counters $\chi_y$ and $\chi_z$ are defined similarly. As in previous proofs, the conclusion of part (ii) of the lemma will follow if we show that $j_x = j_y = j_z$.
\vskip0.1in
\noindent To this end, let us substitute the expressions for $x$, $y$ and $z$ into the equation $tx+(1-t)y=z$,  
\begin{align*} 
&t(R j_x+\dz_x-M\chi_x)+(1-t)(R j_y+\dz_y-M\chi_y)=R j_z+\dz_z-M\chi_z; \text{ in other words} \\ 
&M(\chi_z-t\chi_x-(1-t)\chi_y)+
R[t(j_x-j_y)+j_y-j_z]+t\dz_x+(1-t)\dz_y-\dz_z=0. 
\end{align*}
Setting $\mathscr{I}_1:=M(\chi_z-t\chi_x-(1-t)\chi_y)$, $\mathscr{I}_2 :=R[t(j_x-j_y)+j_y-j_z]$ and $\mathscr{I}_3 :=t\dz_x+(1-t)\dz_y-\dz_z$, 
we arrive at the equation
\begin{equation} \mathscr{I}_1 + \mathscr{I}_2 + \mathscr{I}_3 = 0. \label{3-term-eq} \end{equation} 
We will now argue that in order for \eqref{3-term-eq} to hold, each individual term $\mathscr{I}_j$ must vanish. 
\vskip0.1in
\noindent Let us start with $\mathscr{I}_1$. If $\mathscr{I}_1$ is nonzero, that means that $\chi_x$, $\chi_y$ and $\chi_z$ are not all the same, in which case 
\begin{equation} \label{I1-below} 
|\mathscr{I}_1|=M|\chi_z-t\chi_x-(1-t)\chi_y|\ge\min\{t,1-t\}M\ge \eps_0M.
\end{equation} 
On the other hand, 
 \begin{align} |\mathscr{I}_1| = |\mathscr{I}_2 + \mathscr{I}_3| &\leq 2R(N-1) + 1 \notag \\ &< 2RN \leq 2 \Bigl( \frac{2N^{\tau}}{c}\Bigr) N \leq \frac{4}{c} N^{1 + \tau} \leq \frac{4}{c} \Bigl( \frac{c \eps_0 M}{5}\Bigr) = \frac{4 \eps_0}{5} M. \label{I1-above}  \end{align} 
Combining \eqref{I1-below} and \eqref{I1-above} lead to the desired contradiction. We can now rephrase \eqref{3-term-eq} as $\mathscr{I}_2 + \mathscr{I}_3 = 0$. 
\vskip0.1in
\noindent Suppose if possible that $\mathscr{I}_2$ is nonzero. If $j_x = j_y$, then $|\mathscr{I}_2| \geq R$. If $j_x \ne j_y$, suppose without loss of generality that $j_x > j_y$. Let $\rho = \text{gcd}(j_z-j_y, j_x-j_y)$, with $j_z - j_y = \rho q$ and $j_x - j_y = \rho p$, so that $q \in \mathbb Z$, $p \in \mathbb N$. The defining property of $\cec$ then implies  
\begin{align*} |\mathscr{I}_2|=R|j_x-j_y| \bigl| t-\frac{j_z-j_y}{j_x-j_y}\bigr| &= R|j_x - j_y| \bigl| t - \frac{q}{p} \bigr| \\ &> R|j_x - j_y| \frac{c}{p^{1 + \tau}}
>\frac{Rc}{|j_x-j_y|^\tau}
>\frac{Rc}{(N-1)^\tau}. \end{align*}
In either case, we have that 
\[ \frac{Rc}{(N-1)^\tau} < |\mathscr{I}_2| = |\mathscr{I}_3| < 1. \]  
This leads to the inequality $R < \frac{(N-1)^{\tau}}{c}$, which contradicts the definition of $R$. 
It follows that $\mathscr{I}_2 = 0$, or $j_z - j_y = t(j_x - j_y)$. Since $t$ is irrational, this further implies that $j_x=j_y=j_z$, which is the conclusion of Lemma \ref{l-5}(ii).
\end{proof}

\subsection{Proof of Theorem \ref{t-6}}
\begin{proof}
The proof is similar to its predecessors Theorems \ref{t-1} and \ref{t-4}, so we only sketch the details. 
Let $M\in\nn$ be large enough obeying the conclusions of Lemma \ref{l-5}. 
For any $n\in\nn$, we set $L_n:=\#W_M$, $M_n=M$.
Then
$$\lim_{n\to\fz}\frac{\ln M_{n+1}}{\ln(M_1\cdots M_n)}=0 \quad \text{ and } \quad \lim_{n\to\fz}\frac{\ln (L_1\cdots L_n)}{\ln(M_1\cdots M_n)} \geq \frac{1}{1+\tau}.$$
The sequence of random variables $\{\ell_{\mathbf j}:\ \mathbf {j}\in\Sigma^*\}$ is assumed to be independent and identically distributed 
uniformly in $[M]$. The random sets $X_{\mathbf j}:= W_M +\ell_{\mathbf j} \mod M$, for $\mathbf {j} \in\Sigma_n$ clearly satisfy the independence criteria (I) and (II) on page \pageref{indep-criteria}. 
If $\mu$ is as in \eqref{mu},  
then, by Theorem \ref{t-A}, we have that for every $\eps > 0$, 
$$|\wh\mu(k)|\leq C_{\eps} (1 + |k|)^{-\frac{1}{2(1 + \tau)} + \eps}\mbox{ for all }k\in\zz,$$
in other words, $E = \supp(\mu)$ has Fourier dimension at least $\frac{1}{1+\tau}$. 
\vskip0.1in
\noindent The proof that $E$ avoids nontrivial solutions of $tx+(1-t)y=z$ for any
$t\in\cec\cap(\eps_0,1-\eps_0)$ is also similar to before.
Indeed, if $\{x,y,z\}\subseteq E$  is such a solution and if $I_{\mathbf j}$ is the smallest basic interval in the construction containing $\{x,y,z\}$, it would follow from Lemma \ref{l-5}
that $\{x,y,z\}\in I_{\mathbf j, j_{n+1}}$ for some $j_{n+1} \in [M]$, contradicting minimality of the length of $I_{\mathbf j}$.
\end{proof}

\subsection{Proof of Corollary \ref{t-7}} 
The proof is identical to that of Theorem \ref{t-6}, with the same choice of construction parameters. 
The only distinction is that the building block Lemma \ref{l-5} has to be replaced 
by the stronger Lemma \ref{l-6} below. The set provided by this lemma manages to (partially) avoid both $\mathcal E_{c, \tau}$ as in \eqref{badly-approximable} and $\mathscr{A}(p, \kappa)$ as in \eqref{union-pk}.   
\begin{lemma}\label{l-6}
Fix $0 < \tau, c \leq 1$, $\eps_0\in(0,\frac12)$ and $p\in\nn$ with $p\ge2$. Then for all 
sufficiently large $M\in\nn$ and $\kz=\frac1{2pM}$,
there exists a set $Z_M\subseteq [M]$ with the following properties: 
\begin{enumerate}
\item[{\rm (i)}] $\#Z_M>M^{\frac1{1+\tau}} \exp \Bigl[ \frac1{1+\tau}\ln(\frac{cC_0}{6p})-(2p+5)\sqrt{\ln M} \Bigr]$,

\item[{\rm (ii)}]
If $\{x,y,z\}\st \mathbb I(Z_M+\ell)$ is a solution of 
\begin{equation} \label{badapprox+intervals} 
tx+(1-t)y=z\quad\mbox{ for any }
\quad t\in[\cec\cap(\eps_0,1- \eps_0)]\cup \Bigl[\bigcup_{q=1}^{p-1}\Bigl(\frac qp-\kz,\frac qp+\kz \Bigr)\Bigr]
\end{equation} 
 for some $\ell\in [M]$, 
then there exists $j\in Z_M$ such that $\{x,y,z\}\subseteq \mathtt I_{M,j+\ell}$.
\end{enumerate}
\end{lemma}

\begin{proof}
For sufficiently large integers $M$, we set 
\[ N:=\lfz \Bigl(\frac{c \eps_0M}{4p} \Bigr)^{\frac1{1+\tau}}\rfz \quad \text{ and } \quad R :=\lfz\frac {2pN^\tau}c\rfz+1.\] 
Applying Proposition \ref{p-1} with \[ \cg = \bigcup_{q=1}^{p-1}\Bigl\{f : \mathbb R^3 \rightarrow \mathbb R \; \Bigl| \; f(x, y, z) = pz - qx - \bigl(p-q \bigr) y \Bigr\}  \quad \text{ and } \quad A = 2p+1,  \] 
we find a set $Y_{N}\subseteq [N]$ such that
$\#Y _{N}>N \exp \bigl[-(2p+5)\sqrt{\ln N} \bigr]$
and $Y_{N}$ contains no nontrivial solution of 
\begin{align}\label{inf}
tx &+(1-t)y=z\quad\mbox{ for any } t\in \mathtt{Q}_p := \Bigl\{\frac qp: 1 \leq q \leq p-1 \Bigr\}; \text{ indeed} \\ 
\label{inf1}
\inf\Bigl\{|tx &+(1-t)y-z|:\ \text{ distinct } x,y,z \in Y_{N}, t\in \mathtt Q_p \Bigr\}\ge\frac1p.
\end{align}
We now set $Z_M:=R Y_{N} =\{Rx:\ x\in Y_{N}\} \subseteq [M]$. Then
\begin{align*} \#Z_M=\#Y_{N} &>N \exp \bigl[-(2p+5)\sqrt{\ln N} \bigr] \\
&>M^{\frac1{1+\tau}} \exp \Bigl[ \frac1{1+\tau}\ln\Bigl(\frac{c \eps_0}{6p} \Bigr)-(2p+5)\sqrt{\ln M} \Bigr],
\end{align*}
as claimed in part (i) of the lemma.
\vskip0.1in
\noindent Suppose that $\{x,y,z\}$ is a triple of points in $\mathbb I(Z_M + \ell)$ obeying \eqref{badapprox+intervals} for some $t$ in the forbidden class of coefficients. 
We define the quantities $j_x, j_y, j_z \in Y_N$, the remainders $\dz_x, \dz_y, \dz_z\in[0,1)$ and the binary counters $\chi_x, \chi_y, \chi_z \in \{0, 1\}$ exactly as in \eqref{jx-def} and \eqref{chix-def}.
Substituting these expressions for $x, y, z$ into the equation $tx+(1-t)y=z$ and following the exact same steps as in Lemma \ref{l-5},
we arrive at the equation $\mathscr{I}_1 + \mathscr{I}_2 + \mathscr{I}_3 = 0$, 
where 
\[ \mathscr{I}_1:=M(\chi_z-t\chi_x-(1-t)\chi_y), \quad \mathscr{I}_2:=R[t(j_x-j_y)+j_y-j_z], \quad 
\mathscr{I}_3:=t\dz_x+(1-t)\dz_y-\dz_z. \]
As in Lemma \ref{l-5}, our goal is to show that is possible only when $\mathscr{I}_1 = \mathscr{I}_2 = \mathscr{I}_3 = 0$ and $j_x = j_y = j_z$, which immediately yields the conclusion in part (ii) of the lemma.
\vskip0.1in 
\noindent We always have that \begin{equation}  |\mathscr{I}_1| = |\mathscr{I}_2 + \mathscr{I}_3| \leq 2R(N_1)+1 < 2RN < \frac{3p}{c} N^{\tau+1} < \Bigl(\frac{3p}{c}\Bigr) \Bigl(\frac{c \eps_0}{3p} M \Bigr)= \eps_0M. \label{I1-below-2} \end{equation}  
On the other hand, if $\chi_x$, $\chi_y$ and $\chi_z$ are not all identical, then
\begin{equation} |\mathscr{I}_1|=M|\chi_z-t\chi_x-(1-t)\chi_y|\ge\min\{t,1-t\}M\ge \eps_0M. \label{I1-above-2} 
\end{equation} 
Combining \eqref{I1-below-2} and \eqref{I1-above-2} leads to a contradiction, which allows us to conclude that $\chi_x = \chi_y = \chi_z$, or $\mathscr{I}_1 = 0$.  
\vskip0.1in
\noindent We will now show that equation $\mathscr{I}_2 + \mathscr{I}_3 = 0$ forces $\mathscr{I}_2 = \mathscr{I}_3 = 0$ and $j_x = j_y = j_z$. The proof of this naturally splits into two cases, depending on the nature of $t$.
 The first case occurs when  $t\in\cec\cap(\eps_0,1-\eps_0)$. Let us assume $\mathscr{I}_2 \ne 0$ and aim for a contradiction. 
The estimate for $\mathscr{I}_2$ given in Lemma \ref{l-5} lifts verbatim to this context, and gives that  
\begin{align}\label{i21}
\frac{Rc}{(N-1)^{\tau}} < |\mathscr{I}_2| = |\mathscr{I}_3| \leq 1. 
\end{align}
Since $R > 2pN^{\tau}/c$, the left hand side is bounded below by $2p$, leading to a contradiction in the inequality \eqref{i21}. Thus $\mathscr{I}_2 = 0$. Since $t$ is irrational, this is possible only if $j_x = j_y = j_z$. 
\vskip0.1in
\noindent For the second case, suppose that $t$ lies within an open $\kappa$-neighbourhood of $\mathtt Q_p$.
Then there exist
$q\in\{1,\ldots,p-1\}$ and $\zeta \in(-\kz,\kz)$
such that $t=\frac qp+\zeta$. Assume towards a contradiction that $\mathscr{I}_2 \ne 0$, which incidentally also means that $j_x, j_y, j_z$ are not all identical. Since $\{j_x, j_y, j_z\} \subseteq Y_N$, substituting the value of $t$ into the expression for $\mathscr{I}_2$ and invoking \eqref{inf1} yields 
\begin{align*}
|\mathscr{I}_2|
&\geq R\Biggl[\Bigl|\frac qp j_x+\frac{p-q}pj_y-j_z\Bigr|
-\Bigl|\zeta(j_x+j_y-j_z)\Bigr|\Biggr] \\
&\geq \frac{R}{p} - \kappa N \geq \frac{R}{p} - \frac{1}{2p} \geq \frac{R}{2p} \gg 1.
\end{align*}
This provides a contradiction for large $N$, since $|\mathscr{I}_2| = |\mathscr{I}_3| \leq 1$. Thus $\mathscr{I}_2 = 0$, which means that $\{ j_x, j_y, j_z\} \subseteq Y_N$ is a solution of the equation \[ t j_x + (1-t) j_y = j_z; \; \text{ in other words } \; \frac{q}{p} j_x +  (1 - \frac{q}{p}) j_y - j_z = \Bigl(\frac{q}{p} - t \Bigr) (j_x - j_y). \] The left hand side of the last equation, if nonzero, is bounded from below in absolute value by $R/p \gg 1$, by \eqref{inf1}. The absolute value of the right hand side is bounded above by $\kappa N < 1/2$. This leads to another contradiction, from which we deduce that each side of the equation must vanish. Thus $\frac{q}{p} j_x + (1 - \frac{q}{p}) j_y = j_z$, i.e., $\{ j_x, j_y, j_z \}$ obeys one of the equations in \eqref{inf}. The construction of $Y_N$ then implies that the solution must be trivial, i.e., $j_x = j_y = j_z$.   
%
This completes the proof of Lemma \ref{l-6}.
\end{proof}

\Addresses

%
%
%
%
%
%
%
%
%
%
%
%
%

\end{document}